\documentclass[a4paper]{article}
\usepackage{amsmath}
\usepackage{amssymb}
\usepackage{mathrsfs}
\usepackage{color}
\usepackage{amsmath}
\usepackage{amsfonts}
\usepackage{amsthm}
\newcommand{\si}{\sigma}
\newcommand{\la}{\lambda}
\newcommand{\ol}{\overline}

\newcommand{\pa}{\partial}
\newcommand{\La}{\Lambda}
\newcommand{\al}{\alpha}
\newcommand{\be}{\beta}
\newcommand{\de}{\delta}

\newcommand{\ga}{\gamma}
\newcommand{\ka}{\kappa}

\newcommand{\ve}{\varepsilon}
\newcommand{\ze}{\zeta}
\newcommand{\Om}{\Omega}
\newcommand{\cd}{\cdot}

\newcommand{\R}{{\mathbb R}}
\newcommand{\C}{{\mathbb C}}
\newcommand{\Z}{{\mathbb Z}}
\newcommand{\N}{{\mathbb N}}
\newcommand{\T}{{\mathbb T}}

\newcommand{\cE}{{\cal E}}
\newcommand{\tcE}{\widetilde{\cal E}}
\newcommand{\tE}{\widetilde{E}}
\newcommand{\tN}{\widetilde{N}}
\newcommand{\cF}{{\cal F}}

\newcommand{\cM}{{\cal{M}}}

\renewcommand{\(}{\left(}
\renewcommand{\)}{\right)}
\newcommand{\Th}{\Theta}
\renewcommand{\i}{{\rm i}}

\renewcommand{\tt}{\tilde{t}}

\newcommand{\ttx}{\tilde{t}_\xi}

\newcommand{\tZ}{\tilde{Z}}

\newcommand{\ds}[1]{\displaystyle{#1}}
\renewcommand{\th}{\theta}
\newcommand{\tx}{t_\xi}

\newtheorem{theorem}{\bf Theorem}[section]

\newtheorem{lemma}[theorem]{\bf Lemma}
\newtheorem{proposition}[theorem]{\bf Proposition}
\newtheorem{corollary}[theorem]{\bf Corollary}
\theoremstyle{remark}
  \newtheorem{remark}[theorem]{\sc Remark}
\theoremstyle{definition}
  \newtheorem{definition}[theorem]{Definition}
  \newtheorem{example}[theorem]{\sc Example}
\numberwithin{equation}{section}

\begin{document}

\title{On the energy estimates of semi-discrete wave equations with time dependent propagation speed}
\author{
Fumihiko Hirosawa\footnote{Department of Mathematical Sciences, Faculty of Science, Yamaguchi University, Japan; 
e-mail: hirosawa@yamaguchi-u.ac.jp}
}
\date{}
\maketitle

\begin{abstract}
Discretization is a fundamental step in numerical analysis for the problems described by differential equations, and the difference between the continuous model and discrete model is one of the most important problems. 
In this paper, we consider the difference in the effect of the time-dependent propagation speed on the energy estimate of the solutions for the wave equation and the semi-discrete wave equation which is a discretization with respect to space variables. 
\end{abstract}

%%%%%%%%%%%%%%%%%%%%%%%%%%%%%%%%%%%%%%
\section{Introduction}
%%%%%%%%%%%%%%%%%%%%%%%%%%%%%%%%%%%%%%
%

The $d$-dimensional semi-discrete wave equation is given as the following infinite system of second order ordinary differential equations: 
\begin{equation}\label{u0}
\frac{d^2}{dt^2} u(t)[k] 
-a^2 \sum_{j=1}^d \(u(t)[k+e_j] - 2u(t)[k] + u(t)[k-e_j]\) =0,\;\;
  (t,k)\in \R \times  \Z^d, 
\end{equation}
where $u(t)=\{u(t)[k]\}_{k\in \Z^d}$ and 
$a$ is a positive constant, $\{e_j,\ldots,e_d\}$ is the standard basis of $\R^d$. 
\eqref{u0} is a discretization with respect to the space variables for the following $d$-dimensional wave equation: 
\begin{equation}\label{w}
  \pa_t^2 u(t,x) - a^2 \sum_{j=1}^d \pa_{x_j}^2 u(t,x) = 0,\;\;
  (t,x)\in \R \times \R^d, 
\end{equation}
where $a$ describes the propagation speed of the wave. 
One of the most important properties for the wave equation \eqref{w} is the following equality which is called the energy conservation: 
\begin{equation}\label{ec}
\tilde{E}(t):=\int_{\R^d} |\pa_t u(t,x)|^2\,dx
     +a^2 \sum_{j=1}^d \int_{\R^d} |\pa_{x_j} u(t,x)|^2\,dx
  \equiv \tilde{E}(t_0),\;\;
  t,t_0\in \R. 
\end{equation}
However, if the propagation speed $a$ depends on time variable, then \eqref{ec} does not hold in general. 
On the contrary, the existence of a solution in the Sobolev space may not be valid if $a(t)$ has singularities like non-Lipschitz continuity or degeneration (see \cite{CDS,CJS}); 
thus time dependent propagation speed is possible to give a crucial influence to the property of the wave equation. 

Let us consider the following Cauchy problem for the wave equation with time dependent propagation speed: 
\begin{equation}\label{wC}
\begin{cases}
\ds{\pa_t^2 u(t,x) - a(t)^2 \sum_{j=1}^d \pa_{x_j}^2 u(t,x) = 0}, 
  & (t,x)\in \R_+ \times \R^d, \\
u(0,x)=u_0(x),\;\; (\pa_t u)(0,x)=u_1(x), & x\in \R^d, 
\end{cases}
\end{equation}
where $\R_+:=[0,\infty)$ and 
$a(t)$ satisfies 
\begin{equation}\label{a0a1}
  a_0 \le a(t) \le a_1
\end{equation}
for some positive constants $a_0$ and $a_1$. 
Since the energy conservation does not hold for \eqref{wC}, 
we introduce the generalized energy conservation, which is abbreviated as GEC, 
as the following uniform equivalence of the energy with respect to $t$: 
\begin{equation}
\label{tGEC}
  \tilde{E}(0) \lesssim \tilde{E}(t) \lesssim \tilde{E}(0),
  \;\; t \in \R_+,
\end{equation} 
where $f \lesssim g$ with positive functions $f$ and $g$ denotes that 
there exists a positive constant $C$ such that $f \le Cg$. 
We will use the notations $g \gtrsim f$ and $f \simeq g$ if 
$f\lesssim g$ and $g\lesssim f \lesssim g$ hold, respectively. 
Moreover, we will use $C$ to denote a generic positive constant. 

Let $a\in C^m(\R_+)$ with $m\ge 1$ and $(u_0,u_1)\in H^2\times H^1$. 
Then a unique time global classical solution of \eqref{wC} exists, and the following energy estimate holds: 
\begin{equation}\label{tE-a'L1}
  \tilde{E}(t) \le \exp\(\frac{2}{a_0}\int^t_0 |a'(s)|\,ds\) \tilde{E}(0),
  \;\; t\in \R_+. 
\end{equation}
It follows that GEC holds if $a'\in L^1(\R_+)$. 
On the other hand, the success or failure of GEC depends on the properties of $a(t)$ and the initial data if $a'\not\in L^1((0,\infty))$. 
Indeed, the following result is known: 

%%%%%%%%%%%%%%%%%%%%%%
\begin{theorem}[\cite{RS}]\label{Thm-RS}
\begin{itemize}
\item[(i)] 
If $a\in C^2(\R_+)$, $|a'(t)|\lesssim (1+t)^{-1}$ and $|a''(t)|\lesssim (1+t)^{-2}$, then GEC is established for the Cauchy problem \eqref{wC}. 
\item[(ii)] 
For any positive and monotone increasing function $\nu(t)$ satisfying 
$\lim_{t\to\infty}\nu(t)=\infty$, the conditions 
$|a'(t)|\lesssim \nu(t)(1+t)^{-1}$ and $|a''(t)|\lesssim \nu(t)(1+t)^{-2}$ 
does not necessarily conclude GEC. 
\end{itemize}
\end{theorem}
%%%%%%%%%%%%%%%%%%%%%%

%%%%%%%%%%%%%%%%%%%%%%
\begin{remark}
\eqref{tE-a'L1} is derived by the estimate
\begin{equation}\label{tE'}
  \tilde{E}'(t)=2a'(t)a(t) \sum_{j=1}^d \int_{\R^d} |\pa_{x_j} u(t,x)|^2\,dx
  \le \frac{2|a'(t)|}{a(t)} \tilde{E}(t)
\end{equation}
and Gronwall's inequality. 
We observe from the first equality of \eqref{tE'} that $\tilde{E}(t)$ increases and decreases if $a'(t)>0$ and $a'(t)<0$, respectively. 
That is, time dependent propagation speed causes increase or decrease for the energy. 
Furthermore, we notice that the second inequality is not taken account the sign of $a'(t)$. 
Actually, \eqref{tE'} is obtained with assuming both $a'(t)>0$ and $a'(t)<0$ increase the energy, but Theorem 1.1 (i) is derived with considering some cancellation of the energy which is caused by changing sign of $a'(t)$. 
\end{remark}
%%%%%%%%%%%%%%%%%%%%%%

According to Theorem \ref{Thm-RS}, the oscillation speed of $a(t)$, which is described by the order of $|a'(t)|$, is crucial for GEC, and the order of threshold is $(1+t)^{-1}$. 
%According to \cite{Yg}, the oscillation of $a(t)$ is very slow if $|a'(t)|\lesssim (1+t)^{-1}$ holds. On the other hand, the oscillation of $a(t)$ is fast (or very fast) if it is not very slow, and one cannot generally expect \eqref{tGEC} in this case. 
However, GEC is not determined only the order of $|a'(t)|$. 
Indeed, under the additional assumptions to $a(t)$ below, GEC is possible even if $|a'(t)|\lesssim (1+t)^{-1}$ does not hold. 

\begin{description}
\item[(H1)]
There exists a positive constant $a_\infty$ such that either of the following estimates hold: 
\begin{equation}\label{stbc-}
  \int^\infty_t|a(s)-a_\infty|\,ds \lesssim (1+t)^{\al}
  \;\text{ for }\;
  \al < 0
\end{equation}
or
\begin{equation}\label{stbc}
  \int^t_0|a(s)-a_\infty|\,ds \lesssim (1+t)^{\al}
  \;\text{ for }\;
  0\le \al \le 1.
\end{equation}
\item[(H2)]
$a \in C^m(\R_+)$ with $m\ge 1$ and the following estimates hold for $\be<1$: 
\begin{equation}\label{akc}
  \left|a^{(k)}(t)\right| \lesssim (1+t)^{-k\be},\;\; k=1,\ldots,m. 
\end{equation}
\end{description}

Then the following theorem is established: 
%%%%%%%%%%%%%%%%%%%%%%%%%%%%%%
\begin{theorem}[\cite{H07}]\label{Thm-H07}
If $m\ge 2$, $a(t)$ satisfies {\rm (H1)} and {\rm (H2)} for 
\begin{equation}\label{albem}
  \be \ge \al + \frac{1-\al}{m},
\end{equation}
then GEC is established for the Cauchy problem \eqref{wC}. 
On the other hand, if $\be<\al$, then GEC does not hold in general. 
\end{theorem}
%%%%%%%%%%%%%%%%%%%%%%%%%%%%%%

\begin{remark}
The right hand side of \eqref{albem} is smaller as $m$ larger or $\al$ smaller. That is, the restriction on $\be$ can be weaker if the differentiability of $a(t)$ is higher or the restriction of (H1) is stronger. 
Actually, the Theorem \ref{Thm-H07} does not conclude the optimality of the condition \eqref{albem}, but the limit case $m=\infty$ is nearly optimal for GEC. 
\end{remark}

\begin{remark}
Since the estimate \eqref{stbc} with $\al=1$ is trivial, 
Theorem \ref{Thm-RS} can be considered a special case of Theorem \ref{Thm-H07} without (H1). 
\end{remark}

Let us consider the following initial boundary value problem: 
\begin{equation}\label{wIBV}
\begin{cases}
\ds{\pa_t^2 u(t,x) - a^2 \sum_{j=1}^d \pa_{x_j}^2 u(t,x) = 0}, 
  & (t,x)\in \R_+ \times \Om, \\
u(t,x)=0, & (t,x)\in \R_+\times \pa\Om,\\
u(0,x)=u_0(x),\;\; \(\pa_t u\)(0,x)=u_1(x), & x\in \pa\Om, 
\end{cases}
\end{equation}
where $\Om$ is a bounded domain of $\R^d$ with smooth boundary $\pa\Om$. 
Then the following energy conservation corresponding to \eqref{ec} for \eqref{wC} is established:  
\begin{equation*}%\label{ecb}
\tilde{E}(t):=\int_{\Om} |\pa_t u(t,x)|^2\,dx
     +a^2 \sum_{j=1}^d \int_{\Om} |\pa_{x_j} u(t,x)|^2\,dx
\equiv \tilde{E}(t_0). 
\end{equation*}
Moreover, if $a$ depends on time variable, then the following theorem, which is corresponding to Theorem \ref{Thm-H07} in bounded domain, is established: 
%%%%%%%%%%%%%%%%%%%%%%%%%%%%%%
\begin{theorem}[\cite{H10}]\label{Thm-H10}
If $m\ge 2$ and $a=a(t)$ satisfies {\rm (H2)} for 
\begin{equation}\label{albem-bdd}
  \be \ge \frac{1}{m},
\end{equation}
then GEC is established for the initial boundary problem \eqref{wIBV}. 
\end{theorem}
%%%%%%%%%%%%%%%%%%%%%%%%%%%%%%

The condition \eqref{albem-bdd} in Theorem \ref{Thm-H10} corresponds to \eqref{albem} in Theorem \ref{Thm-H07} with $\al=0$. 
This implies that GEC is established without the assumption (H1) even though the oscillation speed is faster on the problem \eqref{wIBV}. 
Briefly, (H1) and (H2) are required for the estimate of the solution in the time-frequency space for the low and the high frequency part, respectively. 
However, (H1) is not necessary for the problem \eqref{wIBV} because the influence to the low frequency part can be neglected since the largest Dirichlet eigenvalue of the Laplace operator $\sum_{j=1}^d \pa_{x_j}^2$ is strictly negative. 
One of our main interest of the present paper is how the properties (H1) and (H2) for the time dependent propagation speed $a(t)$ relate to the energy estimate for semi-discrete wave equation \eqref{u0}. 

There are many studies on the discretized wave equations with constant propagation speed, but not many results are known for time dependent propagation model, especially in the case that the propagation speed $a(t)$ is singular to collapse GEC. In \cite{CR}, an approximation of the discretized wave equation with respect to time variable is studied in the case that $a(t)$ is degenerate and oscillating which was studied in \cite{CJS}. The result is not directly related to the studies of the present paper; indeed, no approximation to the continuous model will be discussed, but discretization is a useful approach to study the influence of the time dependent propagation speed to the energy of the solution in time-frequency spaces.

%%%%%%%%%%%%%%%%%%%%%%%%%%%%%%%%%%%%%%
%%%%%%%%%%%%%%%%%%%%%%%%%%%%%%%%%%%%%%
%
\section{Discretization and discrete-time Fourier transformation}
%
%%%%%%%%%%%%%%%%%%%%%%%%%%%%%%%%%%%%%%
%%%%%%%%%%%%%%%%%%%%%%%%%%%%%%%%%%%%%%

Let us consider a discretized model of the Cauchy problem for the wave equation \eqref{wC} with respect to the space variables $x$. 

For $f = \{f[k]\}_{k\in \Z^d}$ and $j=1,\ldots,d$, we denote 
the forward and the backward difference operators $D_j^+$ and $D_j^-$ by 
\begin{equation*}
  D_j^+ f := \{f[k+e_j]-f[k]\}_{k\in\Z^d}
  \;\text{ and }\;
  D_j^- f := \{f[k]-f[k-e_j]\}_{k\in\Z^d},
  \;\; k\in \Z^d,
\end{equation*}
respectively. 
%For $D^+ = (D_1^+,\ldots,D_d^+)$ and $D^- = (D_1^-,\ldots,D_d^-)$ 
Then the discrete Laplace operator in $\Z^d$ is given by 
$\sum_{j=1}^d D_j^+ D_j^-$, that is, 
\begin{equation*}
  \sum_{j=1}^d D_j^+ D_j^- f[k]
 =\sum_{j=1}^d \(f[k+e_j]-2f[k]+f[k-e_j]\). 
\end{equation*}

For the solution $u(t,x)$ of \eqref{wC}, we consider the $d$-dimensional infinite matrix valued function $u(t)=\{u(t)[k]\}_{k\in \Z^d}$ as a sampled of $u(t,x)$ on the lattice $\Z^d$. 
Then a discretized model of \eqref{wC} is given as the following initial value problem for an infinite system of ordinary differential equations, 
which is called the semi-discrete wave equation with time dependent propagation speed: 
\begin{equation}\label{u}
\begin{cases}
\ds{
  u''(t)[k] 
 -a(t)^2 \sum_{j=1}^d D_j^+ D_j^- u(t)[k] =0, }
  & (t,k)\in \R_+ \times \Z^d, 
\\[3mm]
\ds{
  u(0)[k]=u_0[k],\;\; u'(0)[k]=u_1[k], }
  &k\in \Z^d.
\end{cases}
\end{equation}
Then we define the total energy for the solution of \eqref{u} by 
\begin{equation*}%\label{E}
  E(t):=
  \sum_{k\in \Z^d}\left|u'(t)[k]\right|^2
 +a(t)^2 \sum_{j=1}^d \sum_{k\in \Z^d}\left|D_j^+ u(t)[k]\right|^2. 
\end{equation*}
Evidently, the energy conservation $E(t)\equiv E(0)$ is valid if the propagation speed $a$ is a constant, but it does not hold in general for variable propagation speed. 
Therefore, the following energy estimate corresponding to \eqref{tGEC} can be considered: 
\begin{equation}\label{GEC}
  E(t) \simeq E(0).
\end{equation}
Here \eqref{GEC} will be also denoted by GEC. 

It is usual to study the energy estimate of \eqref{wC} and \eqref{wIBV} in the time-frequency spaces by introducing Fourier transformation and Fourier coefficients with respect to the space variables of the solution rather than the solutions $u(t,x)$ themselves. 
Therefore, we introduce the discrete-time Fourier transformation to study the energy estimate of the solution for \eqref{u}. 

%%%%%%%%%%%%%%%%%%%%%%%%%%%%%%%%
\begin{definition}
For $f=\{f[k]\}_{k\in \Z^d} \in l^2(\Z^d)$ we define the discrete-time Fourier transformation $\cF_{\Z^d}[f](\th)$ by
\begin{equation*}
  \cF_{\Z^d}[f](\th):=\sum_{k\in \Z^d} e^{-\i k\cd\th} f[k],
  \;\; \th=(\th_1,\ldots,\th_d) \in \R^d,
\end{equation*}
where $k \cdot \th = \sum_{j=1}^d k_j \th_j$ and $k=(k_1,\ldots,k_d)$. 
We denote that $\cF_{\Z^d}[f]=\hat{f}$ and 
$\sum_{k\in \Z^d}=\sum_{k}$ without any confusion. 
\end{definition}
%%%%%%%%%%%%%%%%%%%%%%%%%%%%%%%%

Since the discrete-time Fourier transformation $\hat{f}(\th)$ is the $d$-dimensional Fourier series with the Fourier coefficient $\{f[-k]\}_{k\in \Z^d}$, 
$\hat{f}$ is a $2\pi$-periodic function in $\R^d$. 
That is, the following equality holds:  
\begin{equation*}
  \hat{f}(\th+2k\pi)=\hat{f}(\th)
\end{equation*}
for any $k\in \Z^d$ and $\th\in \T^d$, where $\T:=[-\pi,\pi]$. 
Therefore, we will restrict the domain of $\cF_{\Z^d}[\cd]$ on $\T^d$. 

Here we introduce some lemmas for the discrete-time Fourier transformation, 
where the proofs will be introduced in Appendix. 
%%%%%%%%%%%%%%%%%%%%%%
\begin{lemma}\label{TDFT}
If $f \in l^1(\Z^d)$ then the following equalities are established 
for $j=1,\ldots,d$: 
\begin{equation}\label{TDFT2}
  \cF_{\Z^d}[D_j^+ f](\th) = \(e^{\i\th_j}-1\) \hat{f}(\th)
\end{equation}
and
\begin{equation}\label{TDFT3}
 \cF_{\Z^d}[D_j^+D_j^- f](\th) 
  = -4\(\sin\frac{\th_j}{2}\)^2\hat{f}(\th).
\end{equation}
\end{lemma}
%%%%%%%%%%%%%%%%%%%%%%

%%%%%%%%%%%%%%%%%%%%%%
\begin{lemma}\label{lemm-Parseval}
If $f \in l^2(\Z^d)$ then %$\hat{f}(\th)\in L^2(\T)$, and 
the following Parseval's type equality is established: 
\begin{equation*}%\label{Parseval}
  \sum_{k\in \Z^d}|f[k]|^2
 =\frac{1}{(2\pi)^d}\int_{\T^d}|\hat{f}(\th)|^2\,d\th. 
\end{equation*}
\end{lemma}
%%%%%%%%%%%%%%%%%%%%%%

For $\th\in \T^d$ we define $\xi=\xi(\th)$ by 
\begin{equation}\label{xith}
  \xi(\th)=(\xi_1(\th_1),\ldots,\xi_d(\th_d)),\;\;
  \xi_j(\th_j):=2\sin\frac{\th_j}{2}
  \;\;(j=1,\ldots,d).
\end{equation}
By the discrete-time Fourier transformation, \eqref{u} is reduced to the following problem: 
%%%%%%%%%%%%%%%%%%
\begin{equation}\label{hu}
\begin{cases}
\partial_t^2 \hat{u}(t,\th) 
+a(t)^2 |\xi(\th)|^2 \hat{u}(t,\th) =0, &
  (t,\th)\in \R_+ \times \T^d, 
\\
  \hat{u}(0,\th)=\hat{u}_0(\th),\;\;
  \(\partial_t \hat{u}\)(0,\th)=\hat{u}_1(\th), 
  & \th\in \T^d,
\end{cases}
\end{equation}
where $\hat{u}(t,\th)=\sum_{k} e^{-\i k\cd\th}u(t)[k]$. 

For the solution $\hat{u}(t,\th)$ of \eqref{hu}, we define the energy density function $\cE(t,\th)$ by 
\begin{equation*}%\label{cE_th}
  \cE(t,\th):=
  \left|\pa_t \hat{u}(t,\th)\right|^2
 +a(t)^2 |\xi(\th)|^2 \left|\hat{u}(t,\th)\right|^2.
\end{equation*}
%Evidently, the energy density $\cE(t,\th)$ is conserved, that is, $\cE(t,\th)\equiv \cE(0,\th)$, if the propagation speed $a$ is a constant. 
By Lemma \ref{lemm-Parseval}, the total energy $E(t)$ is represented by $\cE(t,\th)$ as follows: 
%%%%%%%%%%%%%%%%%%%%%%%%%%%%%%
\begin{lemma}\label{lemm-E-cE}
If $u'(t), D_j^+ u(t)\in l^2(\Z^d)$ $(j=1,\ldots,d)$, 
then the following equality is established: 
\begin{equation}\label{EcE}
  E(t) = \frac{1}{(2\pi)^d}\int_{\T^d} \cE(t,\th)\,d\th.
\end{equation}
\end{lemma}
%%%%%%%%%%%%%%%%%%%%%%%%%%%%%%

In the Cauchy problem of the wave equation \eqref{wC}, we shall call it a continuous model, not only the total energy $\tE(t)$ but also the energy density $\tcE(t,\xi)$: 
\begin{equation*}%\label{ectcE}
  \tcE(t,\xi):=
  \left|\pa_t \cF_{\R^d}[u](t,\xi)\right|^2
 +a^2 |\xi|^2 \left|\cF_{\R^d}[u](t,\xi)\right|^2
\end{equation*}
is conserved if the propagation speed $a$ is a constant, 
where $\cF_{\R^d}[u](t,\xi)$ denotes the Fourier transformation of $u(t,x)$ with respect to the space variable $x$. 
However, the energy density is not conserved for time dependent propagation speed in general. 
Indeed, time dependent propagation speed has the effect of changing the total energy with respect to $t$, and it considered to be a phenomenon that caused by transition of energy across frequencies. 
Basically, the behavior of $\tcE(t,\xi)$ is determined by the properties of $a(t)$, and the assumptions for $a(t)$ in Theorem \ref{Thm-H07} ensure the estimate $\tcE(t,\xi)\simeq \tcE(0,\xi)$. 
On the other hand, the negative results for \eqref{tGEC} are implied from the unboundedness of $\tcE(t,\xi)$. 
In particular, the non-existence result in the Sobolev space is derived from the increase in the energy in high frequency part which is provided from the non-Lipschitz continuity with very fast oscillation of $a(t)$. 

The energy density $\cE(t,\th)$ for the solution of the semi-discrete model \eqref{hu} is also conserved if the propagation speed $a$ is a constant. 
On the other hand, if the propagation speed is variable, then the estimate 
\begin{equation}\label{GECcE}
  \cE(t,\th) \simeq \cE(0,\th)
\end{equation}
does not hold in general as in the case of continuous model, and thus GEC may not be established. 
It will be natural that the estimate \eqref{GECcE} is established if $a(t)$ satisfies the same assumptions of Theorem \ref{Thm-RS} and Theorem \ref{Thm-H07}. 
However, we may expect to prove \eqref{GECcE} under some weaker assumptions to $a(t)$, because the range of $|\xi|$ for $\tcE(t,\xi)$ is $[0,\infty)$, but the range of $|\xi(\th)|$ for $\cE(t,\xi)$ is $[0,2\sqrt{d}]$. In the other words, unlike the continuous model, the solution cannot have high-frequency energy above a certain level because the solution of semi-discrete model has a finite resolution. 
Here we note that the situation of high-frequency energy for the semi-discrete model is the corresponding to the low-frequency energy for the initial boundary value problem \eqref{wIBV}. 

%%%%%%%%%%%%%%%%%%%%%%%%%%%%%%%%%%%%%%
%%%%%%%%%%%%%%%%%%%%%%%%%%%%%%%%%%%%%%
%
\section{Main results}
%
%%%%%%%%%%%%%%%%%%%%%%%%%%%%%%%%%%%%%%
%%%%%%%%%%%%%%%%%%%%%%%%%%%%%%%%%%%%%%
Let us generalize the properties \eqref{stbc} of (H1) and (H2) by positive and monotone increasing functions $\Th(t)$ and $\Xi(t)$ on $\R_+$ as follows: 

\begin{description}
\item[(H1*)]
%For a monotone increasing function $\Th(t)$ satisfying $\Th(t)=O(t)$ ($t\to\infty$), 
There exists a positive constant $a_\infty$ such that the following estimate holds: 
\begin{equation}\label{stb}
  \int^t_0|a(s)-a_\infty|\,ds \le \Th(t).
\end{equation}
\item[(H2*)]
$a \in C^m(\R_+)$ with $m\ge 1$ and the following estimates hold for some positive constants $C_k$: % and a monotone increasing function $\Xi(t)$: 
\begin{equation}\label{ak}
  \left|a^{(k)}(t)\right| \le C_k \Xi(t)^{-k},\;\; k=1,\ldots,m. 
\end{equation}
\end{description}

For $\Th(t)$ and $\Xi(t)$ we introduce the following conditions corresponding to the conditions $\al\le \be$ and \eqref{albem}: 
\begin{description}
\item[(H3*)] 
There exists a positive constant $C_0$ such that 
\begin{equation}\label{ThXi}
\Th(t) \le C_0 \Xi(t).
\end{equation}
\item[(H4*)]
For $m\ge 2$, $\Xi(t)^{-m}\in L^1(\R_+)$ and the following estimate holds: 
\begin{equation}\label{int_Cm}
\sup_{t\ge 0}\left\{\Th(t)^{m-1}\int_t^\infty \Xi(s)^{-m}\,ds
\right\}<\infty.
\end{equation}
\end{description}

Then our first theorem is given as follows: 

%%%%%%%%%%%%%%%%%%%%%%%%%%%%%%%%%%%%%%%
\begin{theorem}\label{Thm1}
For the initial value problem of semi-discrete wave equation \eqref{u}, GEC is established if any of the following 
{\rm(i)} to {\rm (iii)} is satisfied: 
%%%%%%%%%%
\begin{itemize}
\item[{\rm (i)}]
{\rm (H1*)} with $\sup_{t\ge 0}\{\Th(t)\}<\infty$. 
\item[{\rm (ii)}]
{\rm (H2*)} with $m=1$ and $\Xi(t)^{-1} \in L^1(\R_+)$. 
\item[{\rm (iii)}]
{\rm (H1*)}, {\rm (H2*)}, {\rm (H3*)} and {\rm (H4*)}. 
\end{itemize}
%%%%%%%%%%
%\begin{itemize}
%\item[{\rm(i)}]
%$m=0$, {\rm (H1*)} with $\sup_{t\ge 0}\{\Th(t)\}<\infty$. 
%%
%\item[{\rm(ii)}]
%$m=1$, {\rm (H2*)} and $\Xi(t)^{-1} \in L^1(\R_+)$. 
%%
%\item[{\rm (iii)}]
%$m \ge 2$, {\rm (H1*)}, {\rm (H2*)}, {\rm (H3*)} and {\rm (H4*)}. 
%\end{itemize}
\end{theorem}
%%%%%%%%%%%%%%%%%%%%%%%%%%%%%%%%%%%%%%%

If we restrict ourselves $\Th(t)=(1+t)^\al$ and $\Xi(t)=(1+t)^\be$ with $m\be>1$, then \eqref{stb}, \eqref{ak} and \eqref{int_Cm} 
are the same as \eqref{stbc}, \eqref{akc} and \eqref{albem}, respectively. 
Moreover, \eqref{ThXi} is valid if \eqref{albem} holds. 
Comparing the assumptions of $a(t)$ in Theorem \ref{Thm1} with 
Theorem \ref{Thm-H07} we observe the followings: 

\begin{itemize}
\item
If $\al>0$, then GEC is established under the same assumptions for $a(t)$. 
\item
Though GEC does not hold for \eqref{wC} in general if $a(t)$ is not Lipschitz continuous, Theorem \ref{Thm1} concludes GEC without any assumption of the continuity for $a(t)$ if $\al=0$. 
\end{itemize}

%Let us consider how the functions $\Th(t)$ and $\Xi(t)$ are defined for some examples of $a(t)$. 

%%%%%%%%%%%%%%%%%%%%%%%%%%%%%%%%%%%%%%%
\begin{example}\label{Ex1}
Let $\chi \in C^m(\R_+)$ be a positive and periodic function. 
For non-negative constants $p$, $q$ and $r$ we define $a \in C^m(\R_+)$ by 
\begin{equation*}
  a(t)=1 + (1+t)^{-p} \chi\((1+t)^q \(\log(e+t)\)^r\).
\end{equation*}
Then we see the followings: 
\begin{equation*}
  \int^t_0|a(s)-1|\,ds \le \Th(t) \simeq
  \begin{cases}
  1 & (p>1), \\
  \log(e+t) & (p=1), \\
  (1+t)^{-p+1} & (p<1).
  \end{cases}
\end{equation*}
Moreover, for any $k=1,\ldots,m$ we have 
\begin{align*}
  \left|a^{(k)}(t)\right| 
  \lesssim \: & 
  (1+t)^{-p}\((1+t)^{q-1} \(\log(e+t)\)^{r} \)^k
\\
  \le \: & 
  \((1+t)^{\frac{p}{m}-q+1} \(\log(e+t)\)^{-r} \)^{-k}
\end{align*}
for $q>0$ and 
\begin{align*}
  \left|a^{(k)}(t)\right| 
  \lesssim  \: &
  (1+t)^{-p}\((1+t)^{-1} \(\log(e+t)\)^{r-1} \)^k
\\
  \le \: &
  \((1+t)^{\frac{p}{m}+1} \(\log(e+t)\)^{-r+1} \)^{-k}
\end{align*}
for $q=0$. 
Hence we set 
\begin{equation}
  \Xi(t)=\begin{cases}
    (1+t)^{\frac{p}{m}-q+1} \(\log(e+t)\)^{-r} & (q>0),\\
    (1+t)^{\frac{p}{m}+1} \(\log(e+t)\)^{-r+1} & (q=0). 
  \end{cases}
\end{equation}
By Theorem \ref{Thm1}, we have GEC if any of the following holds: 
\begin{itemize}
\item $p>1$ (by (i)). 
\item $m=1$ and $q<p$ (by (ii)). 
\item $m \ge 2$, $0<p=q$ and $r=0$ (by (iii)).
\item $m \ge 2$, $p=q=0$ and $0<r\le 1$ (by (iii)).
\end{itemize}
\end{example}
%%%%%%%%%%%%%%%%%%%%%%%%%%%%%%%%%%%%%%%

%The next example is introduced in \cite{H07}. 
%%%%%%%%%%%%%%%%%%%%%%%%%%%%%%%%%%%%%%%
\begin{example}%\label{Ex2}
Let $\chi_\ve \in C^m(\R_+)$ with $m \ge 2$ be a $2\pi$-periodic function with a positive small parameter $\ve$ satisfying 
$\chi_\ve(\tau)=0$ near $\tau=0$ and 
\begin{equation}
  \sup_{\ve>0}\max_{\tau\in[0,2\pi]}
  \left\{\frac{1}{\ve}
  \left|\frac{d^k}{d\tau^k}\chi_\ve(\tau)\right|\right\}
  <\infty\quad(k=0,1,\ldots,m).
\end{equation}
For a positive large constant $\eta$ and positive constants $\al$, $\be$ and $\ka$ satisfying 
\begin{equation}\label{pqka}
  \al \le \ka \le 1
  \;\text{ and }\;
  \al+\frac{1-\al}{m} \le \be \le \ka+\frac{\ka-\al}{m},
\end{equation}
we define the sequences $\{t_j\}_{j=1}^\infty$, $\{\ve_j\}_{j=1}^\infty$, $\{\rho_j\}_{j=1}^\infty$ and $\{\nu_j\}_{j=1}^\infty$ by 
\begin{equation*}
  t_j:=\eta^j,\;\;
  \ve_j:=t_j^{\al-\ka},\;\;
  \rho_j:=\eta^{-1}t_j^\ka
  \;\text{ and }\;
  \nu_j:=\left[t_j^{-\be+\ka+\frac{\ka-\al}{m}}+1\right]. 
\end{equation*}
Then we define $a(t)$ by 
\begin{equation*}
  a(t):=\begin{cases}
  \sqrt{1+\chi_{\ve_j}\(\nu_j\rho_j^{-1}(t-t_j)\)} 
    & \text{ for } \; t\in [t_j-\rho_j,t_j+\rho_j]\;\; (j=1,2,\ldots),\\
  1 & \text{ for } \; t\in \R_+ \setminus \bigcup_{j=1}^\infty [t_j-\rho_j,t_j+\rho_j].
\end{cases}
\end{equation*}
where $[\;]$ denotes the Gauss symbol. 
Here we note that 
\begin{align*}
  t_j + \rho_j + \rho_{j+1}
= \eta^j \(1 + \eta^{-j(1-\ka) -1} + \eta^{-(j+1)(1-\ka)}\)
\le 3 \eta^j \le \eta^{j+1} = t_{j+1}
\end{align*}
for $\eta \ge 3$, it follows that 
$t_j + \rho_j \le t_{j+1} - \rho_{j+1}$. 
Let $t\in[t_{j-1}+\rho_{j-1},t_j+\rho_j]$. 
Then we have 
\begin{equation}
  \int^{t}_0 |a(s)-1|\,ds
  \lesssim \sum_{k=1}^j \ve_k \rho_k
  =\eta^{-1} \sum_{k=1}^j \eta^{k \al} \simeq t_j^\al
  \simeq (1+t)^\al
\end{equation}
and
\begin{align*}
  \left|a^{(k)}(t)\right|
  \lesssim \ve_j\(\nu_j \rho_j^{-1}\)^k 
  \simeq t_j^{-k\be-(\ka-\al)\(1-\frac{k}{m}\)}
  \le t_j^{-k\be}
  \simeq (1+t)^{-k\be}, 
\end{align*}
it follows that \eqref{stb}, \eqref{ak} and \eqref{ThXi} are established with 
\begin{equation}
  \Th(t) \simeq (1+t)^\al\;\text{ and }\;
  \Xi(t) = (1+t)^\be.
\end{equation}
Moreover, by \eqref{pqka} and noting $\al+(1-\al)/m>1/m$, we have 
$\Xi(t)^{-m}\in L^1(\R_+)$ and 
\begin{align*}
  \Th(t)^{m-1}\int^\infty_t \Xi(s)^{-m}\,ds
  \lesssim (1+t)^{\al(m-1)-m\be+1} \lesssim 1,
\end{align*}
thus \eqref{int_Cm} is established. 
Therefore, $a(t)$ satisfies (iii) of Theorem \ref{Thm1}. 
\end{example}
%%%%%%%%%%%%%%%%%%%%%%%%%%%%%%%%%%%%%%%

If $a \not\in C^1(\R_+)$, then \eqref{wC} is not $L^2$ well-posed in general, hence the energy estimate $\tE(t) \lesssim \tE(0)$ cannot be expected. 
However, $\tE(t)$ is not necessarily unbounded if \eqref{wC} is not $L^2$ well-posed. 
For example, if $a(t)$ is a H\"older continuous function, then \eqref{wC} is not $L^2$ well-posed in general but the Gevrey well-posed. 
That is, if the initial data are functions of the Gevrey class of suitable order, then the solution is a function of the Gevrey class, too; 
hence $\tE(t)$ is bounded (see \cite{CDS}). 
If $a(t)$ does not satisfy \eqref{albem}, then GEC does not hold in general. 
However, if the initial data is a function of the Gevrey class of suitable order and \eqref{stbc-} holds, then there exists a positive constant $C(u_0,u_1)$, which depends on not only the initial energy $\tE(0)$ but also a norm of the initial data in the Gevrey class, such that the total energy $\tE(t)$ is bounded as follows (see \cite{EFH}): 
\begin{equation}\label{tBE}
  \tE(t) \le C(u_0,u_1). 
\end{equation}

Faster oscillation of $a(t)$, which is described as smaller $\be$ not to satisfy \eqref{albem}, may increase the high frequency part of the energy for the solution of \eqref{wC} and cause GEC to collapse. 
However, since the high frequency part of the initial energy is small with the initial data in the Gevrey class, $\tE(t)$ can stay bounded against the effect from the faster oscillation of $a(t)$. 
Our second theorem concludes a corresponding estimate to \eqref{tBE} for the solution of the semi-discrete wave equation \eqref{u} under some weaker assumptions than that for $a(t)$ in Theorem \ref{Thm1}. 

%%%%%%%%%%%%%%%%%%%%%%%%%%%%%%%%%%%%%%%
%\begin{remark}
%If \eqref{int_Cm} holds, then setting 
%$r=\La(t)=(\Th(t)/\int^\infty_t \Xi(s)^{-m}\,ds)^{\frac{1}{m}}$, 
%we have 
%\begin{align*}
%  \si(r)
% =\frac{\Th\(\La^{-1}(r)\)}{r}
% =\frac{\Th\(\La^{-1}(r)\)}{\La\(\La^{-1}(r)\)}
% =\frac{\Th\(t\)}{\La\(t\)}
% =\(\Th\(t\)^{m-1}\int^\infty_t \Xi(s)^{-m}\,ds\)^{\frac{1}{m}}
% \lesssim 1.
%\end{align*}
%\end{remark}
%%%%%%%%%%%%%%%%%%%%%%%%%%%%%%%%%%%%%%%

For a positive and strictly increasing function $\La(t)$ on $\R_+$ satisfying $\lim_{t\to\infty}\La(t)=\infty$ such that 
\begin{equation}\label{La-infty}
  \frac{\Th(t)}{\La(t)}% \nearrow \infty \; (t\to\infty),
  \text{ is monotone increasing and }
  \lim_{t\to\infty}\frac{\Th(t)}{\La(t)}=\infty,
\end{equation}
we introduce the following conditions that are alternative to (H3*) and (H4*): 
%%%%%%%%%%%
\begin{description}
\item[(H5*)]
There exists a positive constant $C_0$ such that 
\begin{equation*}%\label{LaXi}
  \La(t)\le C_0 \Xi(t). 
\end{equation*}
\item[(H6*)]
For $m\ge 2$, $\Xi(t)^{-m}\in L^1(\R_+)$ and the following estimate holds:  
\begin{equation*}%\label{La}
  \sup_{t\ge 0}\left\{\La(t)^m \Th(t)^{-1}\int^\infty_t \Xi(s)^{-m}\,ds
  \right\}<\infty.
\end{equation*}
%\[  \La(t) \le \(\frac{\Th(t)}{\int^\infty_t \Xi(s)^{-m}\,ds}\)^{\frac{1}{m}}. \]
\end{description}
%%%%%%%%%%%%

Here we note that if $\La(t) \simeq \Th(t)$, then (H5*) and (H6*) are coincide with (H3*) and (H4*), respectively. 
On the other hand, it is possible that (H6*) holds, but (H4*) does not hold if \eqref{La-infty} is valid. 
Our second theorem is given as follows: 

%%%%%%%%%%%%%%%%%%%%%%%%%%%%%%%%%%%%%%%
\begin{theorem}\label{Thm2}
Let $m\ge 2$. 
If $\Th(t)$, $\Xi(t)$ and $\La(t)$ satisfy {\rm (H1*)}, {\rm (H2*)}, \eqref{La-infty}, {\rm (H5*)} and {\rm (H6*)}, then there exists a positive constant $N_0$ such that the following estimate is established: 
\begin{equation*}%\label{cEbdd}
  \cE(t,\th) \lesssim 
  \exp\(2|\xi(\th)| \Th\(\La^{-1}\(\frac{N_0}{|\xi(\th)|}\)\)\)
  \cE(0,\th). 
\end{equation*}
Consequently, denoting 
\begin{equation*}
  U(N,u_0,u_1):=
  \int_{\T^d}
  \exp\(2|\xi(\th)| \Th\(\La^{-1}\(\frac{N}{|\xi(\th)|}\)\)\)
  \cE(0,\th)
%  \(\left|\hat{u}_1(\th)\right|^2 
%    + |\xi(\th)|^2 \left|\hat{u}_0(\th)\right|^2 \)
  \,d\th
\end{equation*}
for $N>0$, we have the followings: 
\begin{itemize}
\item[{\rm (i)}] 
If $U(N,u_0,u_1)<\infty$ holds for any $N\ge 1$, then there exists a positiv constant $N_0$ such that the energy estimate 
\begin{equation}\label{Ebdd}
  E(t) \lesssim U(N_0,u_0,u_1)
\end{equation} 
is established. 
\item[{\rm (ii)}] 
There exists a positive constant $N_0$, and the energy estimate
\eqref{Ebdd} is established for any $(u_0,u_1)$ satisfying 
$U(N_0,u_0,u_1)<\infty$. 
\end{itemize}
%
%
%
 %It follows that Consequently, if the following estimate is valid: 
%\begin{equation}\label{est_Thm2}
%  U(N,u_0,u_1):=
%  \int_{\T^d}
%  \exp\(2|\xi(\th)| \Th\(\La^{-1}\(\frac{N}{|\xi(\th)|}\)\)\)
%  \cE(0,\th)
%  \,d\th 
%  <\infty
%\end{equation}
%for any given real number $N\ge 1$, 
%then there exists a positive constant $N_0$ such that the following estimate is% established:
%\begin{equation*}%\label{Ebdd}
%  E(t) \lesssim U(N_0,u_0,u_1).
%\end{equation*} 
\end{theorem}
%%%%%%%%%%%%%%%%%%%%%%%%%%%%%%%%%%%%%%%

%%%%%%%%%%%%%%%%%%%%%%%%%%%%%%%%%%%%%%%
\begin{remark}
Denoting $t=\La^{-1}(r^{-1})$ for $r>0$, we have 
\begin{equation}\label{La-infty2}
  \mu(r):=r \Th\(\La^{-1}\(\frac{1}{r}\)\)
  =\frac{\Th(t)}{\La(t)}
  \nearrow \infty \quad (r \to +0)
\end{equation}
by \eqref{La-infty} since $\La^{-1}(r^{-1})$ is positive and strictly decreasing with respect to $r$. 
It follows that 
\begin{equation*}
  \lim_{\th \to 0}
  |\xi(\th)|\Th\(\La^{-1}\(\frac{N}{|\xi(\th)|}\)\)
 =\lim_{|\xi|\to 0}N\mu\(\frac{N}{|\xi|}\) = \infty. 
\end{equation*}
On the other hand, if $\Th(t) \lesssim \La(t)$, then we have
\[
  \sup_{\th\in\T^d\setminus\{0\}}\left\{
  |\xi(\th)|\Th\(\La^{-1}\(\frac{N}{|\xi(\th)|}\)\)
  \right\}<\infty, 
\]
it follows that $U(N,u_0,u_1) \simeq E(0)$ by Lemma \ref{lemm-E-cE}. 
Therefore, \eqref{La-infty} is a reasonable assumption for the case that Theorem \ref{Thm1} cannot be applied. 
\end{remark}
%%%%%%%%%%%%%%%%%%%%%%%%%%%%%%%%%%%%%%%

Since \eqref{La-infty} provides 
\[
  \lim_{|\th| \to 0} |\xi(\th)|\Th\(\La^{-1}\(\frac{1}{|\xi(\th)|}\)\)=\infty, 
\]
the condition
\begin{equation}\label{est_Thm2}
U(N,u_0,u_1)<\infty
\end{equation}
in Theorem \ref{Thm2} requires approximately that the $|\xi(\th)|\hat{u}_0(\th)$ and $\hat{u}_1(\th)$ degenerate at $\th=0$ in an appropriate order which is determined by $\Th(t)$ and $\Xi(t)$. 
The following examples of $\Th(t)$, $\Xi(t)$ and $u_1$ provide \eqref{est_Thm2} for $d=1$ and $u_0=0$. 

%%%%%%%%%%%%%%%%%%%%%%%%%%%%%%%%%%%%%%%
\begin{example}\label{Ex1Thm2}
Let $a(t)$ be defined in Example \ref{Ex1} with $m\ge 2$, $p=q=0$, 
$r \ge 1$ and $|\chi| \le 1$. 
Then we have 
\begin{equation*}
  \Th(t) = 1+t, \;\;
  \Xi(t) \simeq (1+t)\(\log(e+t)\)^{-r+1}, 
\end{equation*}
and (H4*) does not hold. 
Let us define $\La(t)$ by 
\begin{align*}
  \La(t):= \: & (1+t)\(\log(e+t)\)^{-r+1}
 = \(\frac{1+t}{(1+t)^{-m+1}\(\log(e+t)\)^{m(r-1)}}\)^{\frac{1}{m}}
\\
 \simeq \: &  \(\frac{\Th(t)}{\int^\infty_t \Xi(s)^{-m}\,ds}\)^{\frac{1}{m}}. 
\end{align*}
Then \eqref{La-infty}, (H5*) are (H6*) are valid. 
Noting that $\La^{-1}(\tau) \simeq \tau (\log \tau)^{r-1}$ for $\tau\ge e$, 
there exists a positive constant $M$ such that 
\begin{align*}
  \exp\(2 |\xi(\th)| \Th\(\La^{-1}\(\frac{N}{|\xi(\th)|}\)\)\)
  \lesssim
  \exp\(2M N \(\log\frac{N}{|\xi(\th)|}\)^{r-1}\)
%  \lesssim
%  \(\sin\frac{\th}{2}\)^{-2M}
\end{align*}
for any $N\ge 2e$ on $\T\setminus\{0\}$. 
For $M_0>0$, we define $u_1=\{u_1[k]\}_{k\in \Z}$ by 
\begin{equation*}%\label{Ex3_psi1}
  u_1[k]:=\frac{1}{2\pi}\int_\T
  \(\sin^2 \frac{\th}{2}\)^{\frac{M_0}{2}} e^{\i k \th}\,d\th, 
\end{equation*}
it follows that
\begin{align*}
  \cE(0,\th)
% =&a(0)|\xi(\th)|^2 \left|\hat{u}_0(\th)\right|^2
 =\(\sin^2 \frac{\th}{2}\)^{M_0}
 =\(\frac{N}{2}\)^{2M_0} 
  \exp\(-2M_0\log \frac{N}{|\xi(\th)|} \). 
\end{align*}
Then we have 
\begin{align*}
  U(N,u_0,u_1)
  \lesssim \int_\T 
  \exp\(2M N\(\log\frac{N}{|\xi(\th)|}\)^{r-1}
       -2M_0\log \frac{N}{|\xi(\th)|} \)
  \,d\th.
\end{align*}
If $1<r<2$ then $U(N,u_0,u_1)<\infty$ for any $N\ge 2e$, hence Theorem \ref{Thm2} (i) can be applied. 
If $r=1$ then $U(N_0,u_0,u_1)<\infty$ for $M_0 \ge MN_0$, hence Theorem \ref{Thm2} (ii) can be applied if $M_0$ is large enough. 
\end{example}
%%%%%%%%%%%%%%%%%%%%%%%%%%%%%%%%%%%%%%%

%%%%%%%%%%%%%%%%%%%%%%%%%%%%%%%%%%%%%%%
\begin{example}\label{Ex2Thm2}
Let $a(t)$ be defined in Example \ref{Ex1} with $m \ge 1$, $0\le p<q<1$ and $r=0$. 
Then we have 
\begin{equation*}
  \Th(t)\simeq (1+t)^{1-p},\;\;
  \Xi(t) \simeq (1+t)^{\frac{p}{m}-q+1}, 
\end{equation*}
and (H4*) dose not hold. 
Let us define $\La(t)$ by 
\begin{equation*}
  \La(t):=(1+t)^{1-q}
 =\(\frac{(1+t)^{1-p}}{(1+t)^{-p-m(1-q)+1}}\)^{\frac{1}{m}}
 \simeq \(\frac{\Th(t)}{\int^\infty_t \Xi(s)^{-m}\,ds}\)^{\frac{1}{m}}. 
\end{equation*}
Then \eqref{La-infty}, (H5*) are (H6*) are valid. 
Noting that $\Th(\La^{-1}(\tau)) \simeq \tau^{(1-p)/(1-q)}$ for $\tau\ge 1$, there exists a positive constant $M_0$ such that
\begin{align*}
  |\xi(\th)| \Th\(\La^{-1}\(\frac{N}{|\xi(\th)|}\)\)
  \le
  M_0 N^{\frac{1-p}{1-q}}|\xi(\th)|^{-\frac{q-p}{1-q}}
\end{align*}
for any $N\ge 2$ on $\T\setminus\{0\}$. 
For $\rho>0$ and $\ka \ge (q-p)/(1-q)$ we define $u_1=\{u_1[k]\}_{k\in \Z}$ by 
\begin{equation*}
  u_1[k]:=\frac{1}{2\pi}\int_{\T}
  \exp\(-\rho\(\sin^2\frac{\th}{2}\)^{-\frac{\ka}{2}}\)
  e^{\i k \th}\,d\th,
\end{equation*}
it follows that
\begin{align*}
  \cE(0,\th)
 =\exp\(-2 \rho \(\sin^2 \frac{\th}{2}\)^{-\frac{\ka}{2}}\)
 =\exp\(-2^{\ka+1}\rho|\xi(\th)|^{-\ka}\). 
\end{align*}
Then we have 
\begin{align*}
  U(N,u_0,u_1)
  \le
  \int_\T \exp\(
  -2|\xi(\th)|^{-\ka}\(
  2^{\ka}\rho-M_0 N^{\frac{1-p}{1-q}}|\xi(\th)|^{\ka-\frac{q-p}{1-q}}
  \)\)\,d\th.
\end{align*}
If $\ka > (q-p)/(1-q)$, then $U(N,u_0,u_1)<\infty$ for any large $N$, hence Theorem \ref{Thm2} (i) can be applied. 
If $\ka = (q-p)/(1-q)$ then $U(N_0,u_0,u_1)<\infty$ for 
$\rho \ge 2^{-\ka}M_0 N_0^{(1-p)/(1-q)}$, hence Theorem \ref{Thm2} (ii) can be applied if $\rho$ is large enough. 
\end{example}
%%%%%%%%%%%%%%%%%%%%%%%%%%%%%%%%%%%%%%%

By Theorem \ref{Thm2}, Example \ref{Ex2-Thm3} and Lemma \ref{Lemma1-Thm3} we have the following corollary: 
%%%%%%%%%%%%%%%%%%%%%%%%%%%%%%
\begin{corollary}
For $\nu>1$ and $\rho>0$ we define the Gevrey classes $\ga_\rho^\nu$ and $\ga_\infty^\nu$ by 
\begin{equation*}
  \ga^\nu_\rho:=\left\{f\in C^\infty\(\T^d\)\;;\;
\sup_{\th\in\T^d}\left\{\left|\pa_\th^\al f(\th)\right|
  \frac{\rho^{|\al|}}{\al!^\nu}\right\}<\infty\right\},
\;\;
  \ga_\infty^\nu:=\bigcup_{\rho>0} \ga_\rho^\nu,
\end{equation*}
respectively. 
Let {\rm (H1*)} and {\rm (H2*)} hold for 
$\Th(t)\lesssim (1+t)^{1-p}$ and $\Xi(t)=(1+t)^{\frac{p}{m}-q+1}$ 
with $m\ge 1$ and $0\le p < q<1$. 
\begin{itemize}
\item[{\rm (i)}]
If $\xi_j \hat{u}_0,\, \hat{u}_1 \in \ga_\infty^\nu$ 
$(j=1,\ldots,d)$ with 
$\nu<(1-p)/(q-p)$ and 
\begin{equation}\label{eq1-cor}
  \(\pa_\th^\al \xi_j \hat{u}_0\)(0) = 0\;\;(j=1,\ldots,d),\;\;
  \(\pa_\th^\al \hat{u}_1\)(0) = 0,\;\;
  \al\in \N_0^d,
\end{equation}
then there exists a positive constant $N_0$ such that $U(N_0,u_0,u_1)<\infty$ and 
\eqref{Ebdd} is established. 
\item[{\rm (ii)}] 
There exist positive constants $\rho$ and $N_0$ such that 
for any $\xi_j \hat{u}_0,\, \hat{u}_1 \in \ga_\rho^\nu$ 
$(j=1,\ldots,d)$ with $\nu=(1-p)/(q-p)$ satisfying \eqref{eq1-cor}, 
$U(N_0,u_0,u_1)<\infty$ and \eqref{Ebdd} is established.  
\end{itemize}
\end{corollary}

Let us consider the conditions for \eqref{est_Thm2} to the initial data $u_0$ and $u_1$ themselves instead of $\hat{u}_0$ and $\hat{u}_1$. 
In order to describe the conditions, we introduce the logarithmic convexity and the associated functions for sequences.
The sequence of positive real numbers $\{M_j\}=\{M_j\}_{j=0}^\infty$ is called logarithmically convex if the following estimate holds: 
\begin{equation*}
  \frac{M_{j}}{M_{j-1}} \le \frac{M_{j+1}}{M_{j}},\;\;
  j\in \N.
\end{equation*}
For a logarithmically convex sequence $\{M_j\}$, we define the associated function $T[\{M_j\}](\tau)$ on $\R_+$ by 
\begin{equation*}
  T[\{M_j\}](\tau):=\sup_{j\in \N}\left\{\frac{\tau^j}{M_j}\right\}. 
\end{equation*}
Then our third theorem is given as follows: 

%%%%%%%%%%%%%%%%%%%%%%%%%%%%%%%%%%%%%%%%%%%
\begin{theorem}\label{Thm3}
Let $\Th(t)$, $\Xi(t)$ and $\La(t)$ satisfy the same assumptions in Theorem \ref{Thm3}, and $\{M_j\}$ be a logarithmically convex sequence 
satisfying 
\begin{equation}\label{eq1-Thm3}
  L\(N,\{M_j\}\):=
  \inf_{\tau \ge 1}\left\{
  \frac{T\left[\left\{\frac{M_j}{j!}\right\}\right](\tau)}
       {\exp\(\tau^{-1}\Th\(\La^{-1}(N \tau\)\)}
  \right\}>0
\end{equation}
for a positive real number $N$. 
\begin{itemize}
\item[{\rm (i)}] 
If $L(N,\{M_j\})>0$ for any $N>0$, and $(u_0,u_1)$ satisfies 
\begin{equation}\label{eq2-Thm3}
  \sum_{k\in \Z^d} k^{\al} D_j^+ u_0[k]=0
  \;\;(j=1,\ldots,d), \;\;
  \sum_{k\in \Z^d} k^\al u_1[k]=0, \;\;
  \al\in \N_0^d:=(\N\cup\{0\})^d
\end{equation}
and 
\begin{equation}\label{eq3-Thm3}
  \sup_{k\in \N^d}\left\{
  \(\sum_{j=1}^d|D_j^+ u_0[k]|+ |u_1[k]|\)|k|^{d+1} T[\{M_j\}]\(\rho^{-1}|k|\)
  \right\} <\infty 
\end{equation}
for a positive constant $\rho$, 
then there exists a positive constant $N_0$ such that 
$U(N_0,u_0,u_1)<\infty$ and the energy estimate \eqref{Ebdd} is established. 
\item[{\rm (ii)}]
If the following estimate holds: 
\begin{equation}\label{eq4-Thm3}
  \lim_{N\to\infty} \inf_{\tau\ge \frac{1}{2\sqrt{d}}}
  \left\{\frac{\Th\(\La^{-1}(N\tau)\)}{N\Th\(\La^{-1}(\tau)\)}\right\}
  =\infty,
\end{equation}
then there exist positive constants $\rho$ and $N_0$ such that 
for any $(u_0,u_1)$ satisfying \eqref{eq2-Thm3} and \eqref{eq3-Thm3}, 
$U(N_0,u_0,u_1)<\infty$ and \eqref{Ebdd} is established.  
\end{itemize}
\end{theorem}
%%%%%%%%%%%%%%%%%%%%%%%%%%%%%%%%%%%%%%%%%%%

Let us introduce some examples of the choice of $\{M_j\}$ in Theorem \ref{Thm3}. 
\begin{example}%\label{Ex1-Thm3}
Let $a(t)$ be define in Example \ref{Ex1} and Example \ref{Ex1Thm2}. 
If $\{M_j\} = \{j! \exp(b j^{\si})\}$ for $b>0$ and $\si>1$, then there exists a positive constant $b^\ast$ such that 
\begin{align*}
  T\left[\left\{\frac{M_j}{j!}\right\}\right](\tau)
  \gtrsim \exp\(b^\ast \(\log\tau\)^{\frac{\si}{\si-1}}\)
\end{align*}
for any $\tau\ge 1$ (see \cite{M}). 
By the consideration of Example \ref{Ex1Thm2}, there exists a positive constant $M$ such that
\begin{align*}
  \frac{T\left[\left\{\frac{M_j}{j!}\right\}\right](\tau)}
       {\exp\(\tau^{-1}\Th\(\La^{-1}(N \tau\)\)}
  \gtrsim 
  \exp\(b^\ast \(\log\tau\)^{\frac{\si}{\si-1}}
       -M\log(N\tau)^{r-1}\). 
\end{align*}
Therefore, for any $N>0$, 
\eqref{eq1-Thm3} is valid for $\si>1$ if $r\le 2$, and for $(r-1)/(r-2)>\si>1$ if $r>2$. 
Here we note that \eqref{eq4-Thm3} does not hold. 
\end{example}

\begin{example}\label{Ex2-Thm3}
Let $a(t)$ be define in Example \ref{Ex1} and Example \ref{Ex2Thm2}. 
If $\{M_j\} = \{j!^{\nu}\}$ for $\nu>1$, then there exists a positive constant $q$ such that 
\begin{align*}
  T\left[\left\{\frac{M_j}{j!}\right\}\right](\tau)
  \gtrsim \tau^{-\frac12}\exp\((\nu-1) \tau^{\frac{1}{\nu-1}}\)
\end{align*}
for any $\tau\ge 1$ (see \cite{M}). 
By the consideration of Example \ref{Ex1Thm2}, there exists a positive constant $M$ such that
\begin{align*}
  \frac{T\left[\left\{\frac{M_j}{j!}\right\}\right](\tau)}
       {\exp\(\tau^{-1}\Th\(\La^{-1}(N \tau)\)\)}
  \gtrsim 
  \tau^{-\frac12}\exp\((\nu-1) \tau^{\frac{1}{\nu-1}}
       -M N^{\frac{1-p}{1-q}}\tau^{\frac{q-p}{1-q}}\). 
\end{align*}
Therefore, \eqref{eq1-Thm3} is valid for any $N>0$ with $\nu<(1-p)/(q-p)$, 
that is, $1/(\nu-1)>(q-p)/(1-q)$, and thus Theorem \ref{Thm3} (i) can be applied. 
If $\nu=(1-p)/(q-p)$ then \eqref{eq1-Thm3} is valid for 
$N \le ((\nu-1)/M)^{(1-q)/(1-p)}$. 
Noting that $\Th(\La^{-1}(N\tau))/(N\Th(\La^{-1}(\tau)))\simeq N^{(q-p)/(1-q)}\to\infty$ as $N\to\infty$, Theorem \ref{Thm3} (ii) can be applied. 
\end{example}
%%%%%%%%%%%%%%%%%%%%%

%%%%%%%%%%%%%%%%%%%%%%%%%%%%%%%%%%%%%%%%%%%%%
%%%%%%%%%%%%%%%%%%%%%%%%%%%%%%%%%%%%%%%%%%%%%
%
\section{Proof of Theorem \ref{Thm1}}
%
%%%%%%%%%%%%%%%%%%%%%%%%%%%%%%%%%%%%%%%%%%%%%
%%%%%%%%%%%%%%%%%%%%%%%%%%%%%%%%%%%%%%%%%%%%%
%
%%%%%%%%%%%%%%%%%%%%%%%%%%%%%%%%%%%%%%%%%%%%%
%
\subsection{Proof of Theorem \ref{Thm1} (i)}
%
%%%%%%%%%%%%%%%%%%%%%%%%%%%%%%%%%%%%%%%%%%%%%
Let $\xi=\xi(\th)$ be defined by \eqref{xith}. 
Denoting 
\begin{equation*}
  v=v(t,\xi)=\hat{u}(t,\th),
\end{equation*}
the equation of \eqref{hu} is represented as follows: 
\begin{equation}\label{v}
\partial_t^2 v(t,\xi) + a(t)^2 |\xi|^2 v(t,\xi) =0, \;\;
  (t,\xi)\in \R_+ \times [-2,2]^d. 
\end{equation}
Here we denote $\cE(t,\th)$ by $\cE(t,\xi)$: 
\begin{align*}
  \cE(t,\xi) = |\pa_t v(t,\xi)|^2 + a(t)^2 |\xi|^2 |v(t,\xi)|^2 
\end{align*}
without any confusion. 
We define $\cE_\infty(t,\xi)$ by 
\begin{equation}
  \cE_\infty(t,\xi):=|\pa_t v(t,\xi)|^2 + a_\infty^2 |\xi|^2 |v(t,\xi)|^2. 
\end{equation}
Then we have 
\begin{align*}
  \pa_t \cE_\infty(t,\xi)
=\: &2\(a_\infty^2-a(t)^2\)|\xi|^2 \Re\(\pa_t v\, \ol{v}\) 
\\
\le\: &
  \frac{\left|a_\infty^2-a(t)^2\right||\xi|}{a_\infty} 
  \cE_\infty(t,\xi)
\le
  \frac{2a_1}{a_0} \left|a(t)-a_\infty\right||\xi|
  \cE_\infty(t,\xi).
\end{align*}
By Gronwall's inequality and \eqref{stb}, we have 
\begin{align*}
  \cE_\infty(t,\xi)
  \le\:& \exp\(\frac{2a_1}{a_0}\Th(t)|\xi|\) \cE_\infty(0,\xi)
\\
  \le\:& \exp\(\frac{4\sqrt{d} \, a_1}{a_0}\sup_{t\ge 0}\{\Th(t)\}\) 
  \cE_\infty(0,\xi) 
  \simeq \cE_\infty(0,\xi) 
\end{align*}
for any $(t,\xi)\in \R_+ \times [-2,2]^d$. 
Analogously, we have 
\begin{equation*}
  \cE_\infty(t,\xi) \ge 
  \exp\(-\frac{4\sqrt{d} \, a_1}{a_0}\sup_{t\ge 0}\{\Th(t)\}\) 
  \cE_\infty(0,\xi)
  \simeq \cE_\infty(0,\xi). 
\end{equation*}
Therefore, by Lemma \ref{lemm-E-cE} and noting the estimate: 
\begin{equation}\label{cEinfcE}
  \cE_\infty(t,\xi) \simeq \cE(t,\xi)
\end{equation}
due to \eqref{a0a1}, we have 
\begin{equation*}
  E(t) \simeq \int_{\T^d} \cE_\infty(t,\xi(\th))\,d\th
  \simeq  \int_{\T^d} \cE_\infty(0,\xi(\th))\,d\th
  \simeq E(0).
\end{equation*}
\qed

\begin{remark}
It is crucial for the proof of Theorem \ref{Thm1} (i) that $\sup_{\th\in\T}\{|\xi(\th)|\}<\infty$ which comes from the characteristics of the discrete model. 
\end{remark}

%%%%%%%%%%%%%%%%%%%%%%%%%%%%%%%%%%%%%%%%%%%%%
%
\subsection{Proof of Theorem \ref{Thm1} (ii)}
%
%%%%%%%%%%%%%%%%%%%%%%%%%%%%%%%%%%%%%%%%%%%%%

By \eqref{ak} we have 
\begin{align*}
  \pa_t \cE(t,\xi)
 = 2a'(t)a(t)|\xi|^2 |v(t,\xi)|^2
%\\
% \le\:& \frac{2C_1 \Xi(t)^{-1}}{a(t)} a(t)^2|\xi|^2 |v(t,\xi)|^2
 \le \frac{2C_1}{a_0}\Xi(t)^{-1} \cE(t,\xi). 
\end{align*}
Therefore, by Gronwall's inequality we have 
\begin{equation*}
  \cE(t,\xi)
  \le \exp\(\frac{2C_1}{a_0}\left\|\Xi(\cd)^{-1}\right\|_{L^1(\R_+)}\) 
  \cE(0,\xi)
  \simeq \cE(0,\xi). 
\end{equation*}
Analogously, we have 
\begin{equation*}
  \cE(t,\xi) \ge 
  \exp\(-\frac{2C_1}{a_0}\left\|\Xi(\cd)^{-1}\right\|_{L^2(\R_+)}\) 
  \cE(0,\xi)
  \simeq \cE(0,\xi).  
\end{equation*}
Thus the proof is concluded by using Lemma \ref{lemm-E-cE}. 
\qed

%%%%%%%%%%%%%%%%%%%%%%%%%%%%%%%%%%%%%%%%%%%%%
%
\subsection{Proof of Theorem \ref{Thm1} (iii)}
%
%%%%%%%%%%%%%%%%%%%%%%%%%%%%%%%%%%%%%%%%%%%%%
%%%%%%%%%%%%%%%%%%%%%%%%%%%%%%%%%%%%%%%%%%%%%
\subsubsection{Zones}
%%%%%%%%%%%%%%%%%%%%%%%%%%%%%%%%%%%%%%%%%%%%%
We can suppose that $\lim_{t\to\infty}\Th(t)=\infty$; otherwise we have GEC by Theorem \ref{Thm1} (i). 
Let $N$ be a large constant depends on $a_0$, $a_1$ and $C_k$ ($k=0,\cdots,m$), to be chosen later. 
We define $T_0$ and $t_\xi$ for $\xi \in [-2,2]^d$ by 
\[
  T_0:=\max\left\{t\ge 0\;;\; \Th(t) = \frac{N}{2\sqrt{d}}\right\}
\]
and
\[
  t_\xi:=\max\left\{t\ge T_0\;;\; \Th(t)|\xi| = N\right\},
\]
respectively. 
Then we divide the region $\R_+\times [-2,2]^d$ by $(t_\xi,\xi)$ into the pseudo-differential zone $Z_\Psi$ and the hyperbolic zone $Z_H$ as follows: 
\begin{equation*}%\label{ZPsi}
  Z_\Psi:=\left\{(t,\xi)\in \R_+\times [-2,2]^d\;;\; t \le t_\xi \right\}
\end{equation*}
and
\begin{equation*}%\label{ZH}
  Z_H:=\left\{(t,\xi)\in \R_+\times [-2,2]^d \;;\; t \ge t_\xi \right\},
\end{equation*}
respectively. 
We shall estimate $\cE(t,\xi)$ in each zones by different methods. 

%%%%%%%%%%%%%%%%%%%%%%%%%%%%%%%%%%%%%%%%%%%%%
\subsubsection{Estimate in $Z_\Psi$}
%%%%%%%%%%%%%%%%%%%%%%%%%%%%%%%%%%%%%%%%%%%%%
Let $(t,\xi) \in Z_\Psi$, that is, $\Th(t)|\xi| \le N$. 
By the same way as the method for the proof of Theorem \ref{Thm1} (i), we have
\begin{align}\label{estcE0-ZP}
  \cE_\infty(t,\xi)
  \le \exp\(\frac{2a_1}{a_0}\Th(t)|\xi|\) \cE_\infty(0,\xi)
  \le \exp\(\frac{2a_1 N}{a_0}\) \cE_\infty(0,\xi) 
\end{align}
and
\begin{align*}
  \cE_\infty(t,\xi)
  \ge \exp\(-\frac{2a_1 N}{a_0}\) \cE_\infty(0,\xi). 
\end{align*}
Therefore, by \eqref{cEinfcE} we have 
\begin{equation}\label{estcE-ZP}
  \cE(t,\xi)\simeq \cE(0,\xi), \;\; t \le t_\xi.
\end{equation}

%%%%%%%%%%%%%%%%%%%%%%%%%%%%%%%%%%%%%%%%%%%%%
\subsubsection{Estimate in $Z_H$}
%%%%%%%%%%%%%%%%%%%%%%%%%%%%%%%%%%%%%%%%%%%%%

Let us consider the following first order system:
\begin{equation}\label{V}
\pa_t V=A V,\;\;
V=\begin{pmatrix}v_1 \\ v_2 \end{pmatrix},\;\;
A=\begin{pmatrix} \phi(t,\xi) & \ol{r(t,\xi)} \\ 
  r(t,\xi) & \ol{\phi(t,\xi)} \end{pmatrix}. 
\end{equation}
For a complex valued function $\phi$, we denote the real and the imaginary parts of $\phi$ by $\phi_\Re$ and $\phi_\Im$, respectively. 
Then we have the following lemma: 
%%%%%%%%%%%%%%%%%%%%%%%%%%%%
\begin{lemma}\label{lemm-estV}
For any $t,\tilde{t} \in \R$ satisfying $\tilde{t} < t$ 
the solution $V=V(t,\xi)$ of \eqref{V} satisfies the following estimates: 
\begin{equation}\label{estV}
  \|V(t,\xi)\|_{\C^2}^2
  \begin{cases}
  \ds{\le  \exp\(2\int^t_{\tilde{t}} \phi_\Re(s,\xi)\,ds
        + 2\int^t_{\tilde{t}} |r(s,\xi)|\,ds\)
  \|V(t_0,\xi)\|_{\C^2}^2},
\\[3ex]
  \ds{\ge \exp\(2\int^t_{\tilde{t}} \phi_\Re(s,\xi)\,ds
        - 2\int^t_{\tilde{t}} |r(s,\xi)|\,ds\) 
  \|V(t_0,\xi)\|_{\C^2}^2}.
  \end{cases}
\end{equation}

\end{lemma}
\begin{proof}
By Cauchy-Schwarz inequality, we have 
\begin{align*}
  \pa_t \|V\|_{\C^2}^2
&=2\Re\(\pa_t V,V\)_{\C^2}
 =2\Re\(\begin{pmatrix} \phi & 0 \\ 0 & \ol{\phi} \end{pmatrix}V, V\)_{\C^2}
 +2\Re\(\begin{pmatrix} 0 & \ol{r} \\ r & 0 \end{pmatrix}V, V\)_{\C^2}
\\
&=2\phi_\Re\|V\|_{\C^2}^2 + 4\Re\(rv_1\ol{v_2}\)
\begin{cases}
\le 2\(\phi_\Re + |r|\)\|V\|_{\C^2}^2,
\\
\ge 2\(\phi_\Re - |r|\)\|V\|_{\C^2}^2.
\end{cases}
\end{align*}
Thus we have \eqref{estV} by Gronwall's inequality. 
\end{proof}

Since the equation of \eqref{v} is reduced to the following first order system: %
\begin{equation}\label{V1}
  \pa_t V_1=A_1 V_1,\;\;
  V_1:=\begin{pmatrix}
    \pa_t v + \i a(t)|\xi|v \\ \pa_t v -\i a(t)|\xi|v
  \end{pmatrix},
  \;\;
  A_1:=\begin{pmatrix} \phi_1 & \ol{r_1} \\ r_1 & \ol{\phi_1} \end{pmatrix},
\end{equation}
where 
\begin{align*}
  r_1:=-\frac{a'(t)}{2a(t)}=\ol{r_1}
  \;\text{ and }\;
  \phi_1=\frac{a'(t)}{2a(t)}+\i a(t)|\xi|,
\end{align*}
by Lemma \ref{lemm-estV} and noting the equality 
\begin{equation}\label{cEV1}
  \|V_1(t,\xi)\|_{\C^2}^2=2\cE(t,\xi),
\end{equation}
we have 
\begin{equation}\label{estV1}
  \cE(t,\xi)
\begin{cases}
\ds{\le
\frac{a(t)}{a(t_0)}
  \exp\(\int^t_{\tilde{t}} \frac{|a'(s)|}{a(s)}\,ds\) 
  \cE(t_0,\xi)}, 
\\[3ex] 
\ds{\ge
\frac{a(t)}{a(t_0)}
  \exp\(-\int^t_{\tilde{t}} \frac{|a'(s)|}{a(s)}\,ds\) 
  \cE(t_0,\xi)}. 
\end{cases}
\end{equation}
Indeed, \eqref{GEC} is concluded from the estimate \eqref{estV1} 
if $\Xi(t)^{-1}\in L^1(\R_+)$, but not for $\Xi(t)\not\in L^1(\R_+)$. 
Therefore, we introduce the following idea, which is called refined diagonalization procedure taking account the properties \eqref{stb} and \eqref{ak} with $m\ge 2$. 
A basic idea of this method was introduced in \cite{Yg}, and improved in \cite{H07} to take the benefit of the property \eqref{stb}. 
The refined diagonalization procedure can be understood to construct a $2 \times 2$ regular matrix $M=M(t,\xi)$ in $Z_H$ satisfies the following properties: 
\begin{itemize}
\item
$M$ is defined in $Z_H$. 
\item
$V:=M V_1$ satisfies $\|V(t,\xi)\|_{\C^2}^2 \simeq \cE(t,\xi)$. 
\item
$V$ is a solution of \eqref{V}. 
\item
$\int^t_{\tilde{t}}\phi_\Re(s,\xi)\,ds$ is bounded in $Z_H$. 
\item
$\int^t_{\tilde{t}}|r(s,\xi)|\,ds$ is bounded in $Z_H$. 
\end{itemize}
Indeed, if there exists such a matrix $M$, then we immediately have \eqref{GEC} from \eqref{estcE-ZP} and \eqref{estV} with $\tilde{t}=t_\xi$.

For $p\in \N_0$ and $q, r \in \Z$ we introduce the symbol-like class 
$S^{(p)}\{q,r\}$ as the set of functions $f(t,\xi)$ satisfy the following estimates in $Z_H$: 
\begin{equation*}%\label{symbol-Cm}
  \left|\pa_t^k f(t,\xi)\right|\lesssim 
  |\xi|^q \Xi(t)^{-r-k}
  \;\; (k=0,\ldots,p).
\end{equation*}
%Here we denote the general functions of $S^{(p)}\{q,r\}$ by $\si^{(p)}\{q,r\}=\si^{(p)}\{q,r\}(t,\xi)$ for convenience. 
Then we have the following usual algebraic properties and a hierarchy of the classes in $Z_H$:
%%%%%%%%%%%%%%%%%%%%%%%%%%%%%%%%%%%%%%%%
\begin{lemma}\label{symbol}
The following properties are established:

\begin{itemize}
\item[\rm (i)]
If $f\in S^{(p)}\{q,r\}$ with $p\ge1$, then $f\in S^{(p-1)}\{q,r\}$ and $\pa_t f\in S^{(p-1)}\{q,r+1\}$.

\item[\rm (ii)]
If $f_1\in S^{(p)}\{q_1,r_1\}$ and $f_2\in S^{(p)}\{q_2,r_2\}$, then
$f_1f_2\in S^{(p)}\{q_1+q_2,r_1+r_2\}$.
\item[\rm (iii)]
If $f\in S^{(p)}\{q,r\}$, then $f \in S^{(p)}\{q+1,r-1\}$.
\item[\rm (iv)]
If $f\in S^{(p)}\{-q,q\}$ with $q\ge 1$, then for any $\ve>0$ there exists a positive constant $N_0$ such that $|f(t,\xi)|\le \ve$ for any $N \ge N_0$. 
\end{itemize}
\end{lemma}
%%%%%%%%%%%%%%%%%%%%%%%%%%%%%%%%%%%%%%%%%%%%%%%%%%%
\begin{proof}
(i) and (ii) are trivial from the definition of $S^{(p)}\{q,r\}$. 
Let $f(t,\xi)\in S^{(p)}\{q,r\}$. %and $(t,\xi)\in Z_H$. 
By (\ref{ThXi}) we have
\begin{align*}
  \left|\pa_t^k f(t,\xi)\right|
  \lesssim\: & |\xi|^q \Xi(t)^{-r-k}
  \le N^{-1} |\xi|^{q+1} \frac{\Th(t)}{\Xi(t)} \Xi(t)^{-(r-1)-k}
\\
  \le\:& N^{-1} C_0 |\xi|^{q+1} \Xi(t)^{-(r-1)-k}
\end{align*}
for $k=0,\ldots,p$. 
It follows that $f(t,\xi)\in S^{(p)}\{q+1,r-1\}$; thus (iii) is proved. 
If $f(t,\xi)\in S^{(p)}\{-q,q\}$, then by \eqref{ThXi} we have 
\begin{align*}
  |f(t,\xi)|
  \le C|\xi|^{-q}\Xi(t)^{-q}
  \le C N^{-q}\(\frac{\Th(t)}{\Xi(t)}\)^{q}
  \le C C_0^q N^{-q} \le \ve
\end{align*}
by putting $N_0=C_0(C\ve^{-1})^{1/q}$. 
\end{proof}

By Lemma \ref{symbol} we have the following lemma: 
%%%%%%%%%%%%%%%%%%%%%%%%%%%%%%%%%%%%%%%%%%%%%%%%%%%
\begin{lemma}\label{symbol2}
There exists a positive constant $N_0$ such that the following properties are established for any $N \ge N_0$:
\begin{itemize}
\item[\rm (i)] 
If $f(t,\xi)\in S^{(p)}\{0,0\}$ and $f(t,\xi)\gtrsim 1$, 
then $1/f(t,\xi)\in S^{(p)}\{0,0\}$. 
\item[\rm (ii)] 
If $f(t,\xi)\in S^{(p)}\{1,0\}$ and $f(t,\xi)\gtrsim |\xi|$, 
then $1/f(t,\xi)\in S^{(p)}\{-1,0\}$. 
\item[\rm (iii)] 
If $f(t,\xi)\in S^{(p)}\{-q,q\}$ with $q \ge 1$, 
then $\sqrt{1-|f(t,\xi)|^2}\in S^{(p)}\{0,0\}$. 
\end{itemize} 
\end{lemma}
%%%%%%%%%%%%%%%%%%%%%%%%%%%%%%%%%%%%%%%%%%%%%%%%%%%
\begin{proof}
Since (i) is proved by the same way as the proof of (ii), we prove (ii). 
%(i) Let $f_0$ be a positive constant satisfying $f(t,\xi)\ge f_0$. 
By applying Fa\`a di Bruno's formula:  
\begin{equation*}%\label{FaB}
  \frac{d^k}{dt^k}F(G(x))
 =k!\sum_{h=1}^k 
  F^{(h)}(G(t))
  \sum_{\substack{h_1 + 2h_2 + \cdots + n h_n=k \\ 
                  h_1+h_2+\cdots + h_k = h}}
  \,
  \prod_{l=1}^k \frac{1}{h_l!\, l!^{h_l}}\(G^{(l)}(t)\)^{h_l} 
\end{equation*}
with $F(G)=1/G$ and $G(t)=f(t,\xi)$, and noting 
$|F^{(h)}(G)|\lesssim 1/|G|^{h+1}$, we have 
\begin{align*}
  \left|\pa_t^k \frac{1}{f(t,\xi)}\right|
\lesssim\:&
  \sum_{h=1}^k \frac{1}{f(t,\xi)^{h+1}}
  \sum_{\substack{h_1 + 2h_2 + \cdots + n h_n=k \\ 
                  h_1+h_2+\cdots + h_k = h}}
  \,
  \prod_{l=1}^k \left|\pa_t^l f(t,\xi)\right|^{h_l}
\\
\lesssim\:&
  \sum_{h=1}^k |\xi|^{-h-1}
  \sum_{\substack{h_1 + 2h_2 + \cdots + n h_n=k \\ 
                  h_1+h_2+\cdots + h_k = h}} \,
  \prod_{l=1}^k \(|\xi| \Xi(t)^{-l}\)^{h_l}
\\
\simeq\:&
  |\xi|^{-1} \Xi(t)^{-k}
\end{align*}
for $k=0,\ldots,p$; thus (ii) is proved.  
By Lemma \ref{symbol} (iv) we can suppose that 
$\sqrt{1-|f(t,\xi)|^2} \ge 1/2$. 
By applying Fa\`a di Bruno's formula with 
$F(G)=\sqrt{1-G}$ and $G(t)=|f(t,\xi)|^2 \in S^{(p)}\{-2q,2q\}$, 
and noting the estimates 
\begin{align*}
  |F^{(h)}(G)| \lesssim \frac{F(G)}{(1-G)^h} \le 4^h
\end{align*}
for $h=0,1,\ldots$, we have 
\begin{align*}
  \left|\pa_t^k \sqrt{1-|f(t,\xi)|^2}\right|
\lesssim\:&
  \sum_{h=1}^k 4^h 
  \sum_{\substack{h_1 + 2h_2 + \cdots + n h_n=k \\ 
                  h_1+h_2+\cdots + h_k = h}}
  \,
  \prod_{l=1}^k \left|\pa_t^l |f(t,\xi)|^2\right|^{h_l}
\\
\lesssim\:&
  \sum_{h=1}^k 
  \sum_{\substack{h_1 + 2h_2 + \cdots + n h_n=k \\ 
                  h_1+h_2+\cdots + h_k = h}}
  \,
  \prod_{l=1}^k \(|\xi|^{-2q}\Xi(t)^{-2q-l}\)^{h_l}
\\
\lesssim\:&
  \sum_{h=1}^k |\xi|^{-2qh}\Xi(t)^{-2qh-k}
\le \Xi(t)^{-k} \sum_{h=1}^k \(\frac{C_0}{N}\)^{2qh}
\\
\simeq\:& \Xi(t)^{-k} 
\end{align*}
for $k=1,\ldots,p$; thus (iii) is prove. 
\end{proof}
%%%%%%%%%%%%%%%%%%%%%%%%%%%%%%%

An eigenvalue of $A_1$ is represented by 
\begin{equation*}
  \la_1=\phi_{1\Re} + \i \phi_{1\Im}\sqrt{1 -\(\frac{|r_1|}{\phi_{1\Im}}\)^2}. 
\end{equation*}
By \eqref{ak}, \eqref{ThXi} and choosing $N>C_0 C_1 a_0^{-2}/2$, we have 
\begin{align*}
  \(\frac{|r_1|}{\phi_{1\Im}}\)^2
\le \(\frac{C_1}{2a_0^2|\xi|\Xi(t)}\)^2
\le \(\frac{C_1}{2a_0^2N}\frac{\Th(t)}{\Xi(t)}\)^2
\le \(\frac{C_0 C_1}{2a_0^2N}\)^2<1.
\end{align*}
It follows that the other eigenvalue of $A_1$ is given by $\ol{\la_1}$. 
Therefore, a diagonalizer $M_1$ for $A_1$ is formally given by 
\begin{equation*}
  M_1=\begin{pmatrix} 1 & \ol{\de_1} \\ \de_1 & 1 \end{pmatrix},
  \quad
  \de_{1}=\frac{\la_{1}-\phi_1}{\ol{r_1}}. 
\end{equation*}
Here we note that the following lemma is established which ensures the invertibility of $M_1$: 
%%%%%%%%%%%%%%%%%%%%%%%%%%%%%%%%%%%%%%%%%%%%%%%%%%%%%
\begin{lemma}\label{symbol_de1}
There exists a positive constant $N_0$ such that 
$r_1\in S^{(m-1)}\{0,1\}$, $1/\phi_{1\Im}\in S^{(m)}\{-1,0\}$
and $\de_1\in S^{(m-1)}\{-1,1\}$ for any $N \ge N_0$. 
\end{lemma}
%%%%%%%%%%%%%%%%%%%%%%%%%%%%%%%%%%%%%%%%%%%%%%%%%%%%%
\begin{proof}
The first two properties are trivial. 
Let $N_0$ be large enough so that Lemma \ref{symbol} and Lemma \ref{symbol2} can be applied. 
By Lemma \ref{symbol} (ii) and (iv), we have 
$(|r_1|/\phi_{1\Im})^2 \in S^{(m-1)}\{-2,2\}$ and 
$(|r_1|/\phi_{1\Im})^2 \le 1/2$. 
By Lemma \ref{symbol2} (iii) we have 
\begin{align*}
  \sqrt{1-\(\frac{|r_1|}{\phi_{1\Im}}\)^2}+1\in S^{m-1}\{0,0\},
%  \; \text{ and } \;
%  \sqrt{1-\(\frac{|r_1|}{\phi_{1\Im}}\)^2} + 1 \ge 1,
\end{align*}
it follows from Lemma \ref{symbol2} (i) that 
\begin{align*}
  \frac{1}{\sqrt{1-\(\frac{|r_1|}{\phi_{1\Im}}\)^2}+1}
  \in S^{(m-1)}\{0,0\}. 
\end{align*}
Therefore, by Lemma \ref{symbol} (ii) and the first two properties of Lemma \ref{symbol_de1}, we have 
\begin{align*}
  \de_1
=\frac{\i \phi_{1\Im}}{\ol{r_1}}
  \(\sqrt{1-\(\frac{|r_1|}{\phi_{1\Im}}\)^2}-1\)
=-\frac{\i r_1}{\phi_{1\Im}}
  \frac{1}{\sqrt{1-\(\frac{|r_1|}{\phi_{1\Im}}\)^2}+1}
\in S^{(m-1)}\{-1,1\}. 
\end{align*}
\end{proof}
%%%%%%%%%%%%%%%

We define $A_2=A_2(t,\xi)$ by
\begin{align*}
  A_2:=\begin{pmatrix} \phi_2 & \ol{r_2} \\ r_2 & \ol{\phi_2} \end{pmatrix},
  \;\;
  r_2:=-\frac{(\de_1)_t}{1-|\de_1|^2},\;\;
  \phi_2:=\la_1+\frac{\ol{\de_1}(\de_1)_t}{1-|\de_1|^2}. 
\end{align*}
Then we see that $r_2 \in S^{(m-2)}\{-1,2\}$ by the lemmas for the symbol classes above. 
Noting the equalities 
\begin{align*}
  M_1^{-1} \(A_1 - \pa_t I\)M_1
=\begin{pmatrix} \la_{1} & 0 \\ 0 & \ol{\la_{1}} \end{pmatrix}
 -M_1^{-1} \(\pa_t M_1\) -\pa_t I
=A_2 - \pa_t I,
\end{align*}
\eqref{V1} is reduced to the following system: 
\begin{equation*}%\label{V2}
  \pa_t V_2 = A_2 V_2,\quad
  V_2:=M_1^{-1}V_1.
\end{equation*}
Here we remark the followings: 
\begin{itemize}
\item 
$M_1$ is a diagonalizer for $A_1$ but not for $A_1-\pa_t I$, 
that is, $A_1 - M_1^{-1}(\pa_t M_1)$ is not diagonal. 
\item 
Since $r_1 \in S^{(m-2)}\{0,1\}$ and $r_2 \in S^{(m-2)}\{-1,2\}$, 
$M_1$ is a diagonalizer for $A_1-\pa_t I$ modulo $S^{(m-2)}\{0,1\}$. 
\item
The structure in which both the diagonal and the off-diagonal entries are complex conjugate are conserved by the diagonalization procedure due to $M_1$. 
\end{itemize}
Therefore, one can carry out the same diagonalization procedure for $A_2-\pa_t I$ if $m\ge 3$. 
Generally, we have the following lemma, which is the essential of the refined diagonalization procedure. 
%%%%%%%%%%%%%%%%%%%%%%%%%%%%%%%%%%%%%%%%%%%%%%%%%%%%%%%%%%%%%
\begin{lemma}\label{ref_diag}
Let $k$ be a positive integer satisfying $k<m$, $A_k$ be given by
\begin{equation*}
  A_k=\begin{pmatrix} \phi_k & \ol{r_k} \\ r_k & \ol{\phi_k} \end{pmatrix},
\end{equation*}
where $r_k$ and $\phi_{k\Im}$ satisfy 
\begin{equation}\label{rphik}
  r_k\in S^{(m-k)}\{-k+1,k\},\;\;
  \phi_{k\Im} \in S^{(m-k)}\{1,0\}
  \;\text{ and }\;
  |\phi_{k\Im}| \gtrsim |\xi|. 
\end{equation}
Then there exists a positive constant $N_k$ such that the following properties are established for any $N \ge N_k$: 
\begin{itemize}
\item[\rm (i)]
$A_k$ has complex conjugate eigenvalues $\la_k$ and $\ol{\la_k}$. 
\item[\rm (ii)]
The following matrix $M_k$ is invertible: 
\begin{equation*}
  M_k:=\begin{pmatrix} 1 & \ol{\de_k} \\ \de_k & 1 \end{pmatrix},
  \quad
  \de_k:=\frac{\la_k-\phi_k}{\ol{r_k}}, 
\end{equation*}
moreover, we have 
\begin{equation}\label{dek}
  \de_k \in S^{(m-k)}\{-k,k\}. 
\end{equation}
\item[\rm (iii)]
$A_{k+1}:=M_k^{-1}A_kM_k - M_k^{-1}(\pa_t M_k)$ is represented as follows: 
\begin{equation*}
  A_{k+1}=\begin{pmatrix}
  \phi_{k+1} & \ol{r_{k+1}} \\ r_{k+1} & \ol{\phi_{k+1}} \end{pmatrix},
\end{equation*}
where $r_{k+1}$ and $\phi_{k+1}=\phi_{(k+1)\Re}+\i\phi_{(k+1)\Im}$ are given by 
\begin{equation}\label{rk+1}
  r_{k+1}=-\frac{(\de_k)_t}{1-|\de_k|^2},
\end{equation}
\begin{equation*}%\label{phi_k+1Re}
  \phi_{(k+1)\Re} =\phi_{k\Re}-\frac{\pa_t \log\(1-|\de_{k}|^2\)}{2} 
\end{equation*}
and
\begin{equation*}%\label{phi_k+1Im}
  \phi_{(k+1)\Im}
 =\phi_{k\Im}\sqrt{1-\(\frac{|r_k|}{\phi_{k\Im}}\)^2}
   -\Im\(\ol{\de_k}r_{k+1}\).
\end{equation*}
\item[\rm (iv)]
$r_{k+1}$ and $\phi_{(k+1)\Im}$ satisfy the followings: 
\begin{equation*}%\label{rphik+1}
  r_{k+1}\in S^{(m-k-1)}\{-k,k+1\},
  \quad
  \phi_{(k+1)\Im} \in S^{(m-k-1)}\{1,0\},
  \;\text{ and }\;
  \phi_{(k+1)\Im} \gtrsim |\xi|.
\end{equation*}
\end{itemize}
\end{lemma}
%%%%%%%%%%%%%%%%%%%%%%%%%%%%%%%%%%%%%%%%%%%%%%%%%%%%%%%%%%%%%
\begin{proof}
Let $N \ge N_k$ with a large constant $N_k$ so that Lemma \ref{symbol} and Lemma \ref{symbol2} can be applied. 
An eigenvalue of $A_k$ is represented by 
\begin{equation*}
  \la_k=\phi_{k\Re}+\i\phi_{k\Im}\sqrt{1-\(\frac{|r_k|}{\phi_{k\Im}}\)^2}. 
\end{equation*}
By \eqref{rphik}, Lemma \ref{symbol} (ii) and Lemma \ref{symbol2} (ii), 
we have 
\begin{equation}\label{rkphik}
  \(\frac{|r_{k}|}{\phi_{k\Im}}\)^2 \in S^{(m-k)}\{-2k,2k\}. 
\end{equation}
It follows from Lemma \ref{symbol} (iv) that the other eigenvalue of 
$A_k$ is given by $\ol{\la_k}$. 
By \eqref{rphik}, \eqref{rkphik} and using the same way as the proof of Lemma \ref{symbol_de1}, we have \eqref{dek}. 
Moreover, $M_k$ is invertible since $|\de_k|^2 \le 1/2$ by Lemma \ref{symbol} (iv). 
(iii) is prove by direct computations. 
By applying Lemma \ref{symbol2} (i) with $f(t,\xi)=1-|\de_k|^2\ge 1/2$, 
we have $(1-|\de_k|^2)^{-1}\in S^{(m-k)}\{0,0\}$. 
Therefore, by Lemma \ref{symbol} (i) and \eqref{rk+1}, we have 
$r_{k+1}\in S^{(m-k-1)}\{-k,k+1\}$, 
hence we have $\phi_{(k+1)\Im} \in S^{m-k-1}\{1,0\}$ 
by Lemma \ref{symbol} (ii), (iii), Lemma \ref{symbol2} (iii), 
\eqref{dek} and \eqref{rkphik}. 
%Finally, let us prove $\phi_{(k+1)\Im} \gtrsim |\xi|$. 
Noting \eqref{rkphik}, $\Im(\ol{\de_k}r_{k+1})/\phi_{k\Im} \in S^{(m-k-1)}\{-2k-1,2k+1\}$ and Lemma \ref{symbol} (iv), we can suppose that 
\begin{align*}
  \(\frac{|r_{k}|}{\phi_{k\Im}}\)^2 \le \frac14
  \;\text{ and }\;
  \frac{\Im\(\ol{\de_k}r_{k+1}\)}{\phi_{k\Im}}\le \frac14. 
\end{align*}
It follows that 
\begin{align*}
  \phi_{(k+1)\Im}
 =\phi_{k\Im}\(
  \sqrt{1-\(\frac{|r_k|}{\phi_{k\Im}}\)^2}
   -\frac{\Im\(\ol{\de_k}r_{k+1}\)}{\phi_{k\Im}}\)
 \ge \frac{\sqrt{3}-1}{2}\phi_{k\Im}
  \gtrsim |\xi|. 
\end{align*}
\end{proof}

By Lemma \ref{ref_diag} we have the following proposition: 
%%%%%%%%%%%%%%%%%%%%%%%%%%%%%%%%%%%%%%%%%%%%%
\begin{proposition}
There exists a positive constant $N_m$ such that the following estimate is established for any $N \ge N_m$ in $Z_H$: 
\begin{equation}
  \cE(t,\xi) \simeq \cE(t_\xi,\xi). 
\end{equation}
\end{proposition}
%%%%%%%%%%%%%%%%%%%%%%%%%%%%%%%%%%%%%%%%%%%%%
\begin{proof}
Let $M_k$ ($k=1,\ldots,m-1$) be defined by Lemma \ref{ref_diag}. 
Then we can suppose that $|\de_k(t,\xi)|^2 \le 1/2$ for $k=1,\ldots,m-1$. 
Denoting $\cM_{m-1} = M_1\cdots M_{m-1}$ and $V_m=\cM_{m-1}^{-1}V_1$, we have 
\begin{equation*}
  \(\pa_t - A_m\) V_m
 =\cM_{m-1}^{-1} \(\pa_t - A_1\)\cM_{m-1} V_m
 =\cM_{m-1}^{-1} \(\pa_t - A_1\) V_1 = 0
\end{equation*}
by \eqref{V1}. 
Therefore, by \eqref{estV} we have 
\begin{equation*}
  \|V_m(t,\xi)\|_{\C^2}^2
\begin{cases}
  \ds{\le \exp\(2\int^t_{t_\xi} \phi_{m\Re}(s,\xi)\,ds
        + 2\int^t_{t_\xi} |r_m(s,\xi)|\,ds\) 
  \|V_m(t_\xi,\xi)\|_{\C^2}^2}, 
\\[3ex]
  \ds{\ge \exp\(2\int^t_{t_\xi} \phi_{m\Re}(s,\xi)\,ds
        - 2\int^t_{t_\xi} |r_m(s,\xi)|\,ds\) 
  \|V_m(t_\xi,\xi)\|_{\C^2}^2}. 
\end{cases}
\end{equation*}
By Lemma \ref{ref_diag} (iii) and (iv) we have 
\begin{align*}
  \exp\(2\int^t_{t_\xi} \phi_{m\Re}(s,\xi)\,ds\)
&=\exp\(2\int^t_{t_\xi} \phi_{1\Re}(s,\xi)\,ds 
  + \sum_{k=1}^{m-1} \log\frac{1-|\de_k(t,\xi)|^2}{1-|\de_k(t_\xi,\xi)|^2}\)
\\
&= \frac{a(t)}{a(t_\xi)}
   \prod_{k=1}^{m-1}\frac{1-|\de_k(t,\xi)|^2}{1-|\de_k(t_\xi,\xi)|^2}
\\
&\begin{cases}
\le
  \(2^{m-1}a_0^{-1} a_1\),
  \\
\ge  \(2^{m-1}a_0^{-1} a_1\)^{-1}
\end{cases}
\end{align*}
and
\begin{align*}
  \exp\(2\int^t_{t_\xi} |r_m(s,\xi)|\,ds\)
\le\:&\exp\(C|\xi|^{-m+1} \int^t_{t_\xi} \Xi(s)^{-m}\,ds \)
\\
\le\:&\exp\(CN^{-m+1}\Th(t_\xi)^{m-1} \int^\infty_{t_\xi} \Xi(s)^{-m}\,ds \)
\simeq 1
\end{align*}
by \eqref{int_Cm}. 
It follows that
\begin{equation}\label{estVm}
  \|V_m(t,\xi)\|_{\C^2}^2 \simeq \|V_m(t_\xi,\xi)\|_{\C^2}^2
\end{equation}
in $Z_H$. 
Therefore, noting the estimates: 
\begin{align*}
  \|V_k\|_{\C^2}^2
 =\|M_k V_{k+1}\|_{\C^2}^2
  \le \(1+|\de_k|\)^2 \|V_{k+1}\|_{\C^2}^2 
  \le \frac{3+2\sqrt{2}}{2}\|V_{k+1}\|_{\C^2}^2
\end{align*}
and
\begin{align*}
  \|V_k\|_{\C^2}^2
\ge \(1-|\de_k|\)^2 \|V_{k+1}\|_{\C^2}^2 
\ge \frac{3-2\sqrt{2}}{2}\|V_{k+1}\|_{\C^2}^2
\end{align*}
for $k=1,\ldots,m-1$, it follows that 
\begin{equation*}%\label{simeqVmV1}
  \(\frac{3 - 2\sqrt{2}}{2}\)^{m-1}\|V_m\|_{\C^2}^2
  \le \|V_1\|_{\C^2}^2
  \le \(\frac{3 + 2\sqrt{2}}{2}\)^{m-1}\|V_m\|_{\C^2}^2,
\end{equation*}
and \eqref{cEV1}, we have 
\begin{equation}\label{estcE-ZH}
  \cE(t,\xi) \simeq \|V_m(t,\xi)\|_{\C^2}^2 
  \simeq \|V_m(t_\xi,\xi)\|_{\C^2}^2
  \simeq \cE(t_\xi,\xi).
\end{equation}
\end{proof}

Combining the estimates \eqref{estcE-ZP}, \eqref{estcE-ZH} and 
using Lemma \ref{lemm-Parseval}, the proof of Theorem \ref{Thm1} (iii) is concluded.

%%%%%%%%%%%%%%%%%%%%%%%%%%%%%%%%%%%%%%%%%%%%%
%
\section{Proof of Theorem \ref{Thm2}}
%
%%%%%%%%%%%%%%%%%%%%%%%%%%%%%%%%%%%%%%%%%%%%%

For a large constant $\tN$, we define $\tilde{T}_0$ and $\tt_\xi$ for $\xi \in [-2,2]^d$ by 
\[
  \tilde{T}_0:=\max\left\{t\ge 0\;;\; \La(t) = \frac{\tN}{2\sqrt{d}}\right\}
  \;\text{ and }\;
  \tt_\xi:=\max\left\{t\ge \tilde{T}\;;\; \La(t)|\xi| = \tN\right\},
\]
respectively. 
Then we define the zones $\tZ_H$ and $\tZ_\Psi$ by 
\begin{equation*}%\label{tZPsi}
  \tZ_\Psi:=\left\{(t,\xi)\in  \R_+\times [-2,2]^d \;;\; t \ge \tt_\xi \right\}
\end{equation*}
and
\begin{equation*}%\label{tZH}
  \tZ_H:=\left\{(t,\xi)\in  \R_+\times [-2,2]^d \;;\; t\le \tt_\xi \right\},
\end{equation*}
respectively. 

Let $t\le \ttx$. 
By the same way to derive the estimate \eqref{estcE0-ZP} and using \eqref{La-infty2}, we have 
\begin{align*}
  \cE(t,\xi)
  \lesssim\:& \exp\(\frac{2a_1}{a_0}\Th\(\ttx\)|\xi|\)\cE(0,\xi)
\\
  =\: &\exp\(\frac{2a_1 \tN}{a_0}\frac{|\xi|}{\tN}
  \Th\(\La^{-1}\(\frac{\tN}{|\xi|}\)\)\)
  \cE(0,\xi) 
\\
  \le\: &\exp\(\frac{2a_1 \tN}{a_0} \frac{|\xi|}{N}
  \Th\(\La^{-1}\(\frac{N}{|\xi|}\)\)\)
  \cE(0,\xi)
\\
  \le\: &\exp\(|\xi|\Th\(\La^{-1}\(\frac{N}{|\xi|}\)\)\)
  \cE(0,\xi) 
\end{align*}
for any $N \ge 2a_1 \tN / a_0$. 

Let $t\ge \ttx$. 
Noting that $\ttx \ge \tx$ with $\tN=N$, Lemma \ref{ref_diag} is valid in $\tZ_H$ for any large $\tN$. 
Therefore, by (H6*) there exist positive constants $C$ and $N$ such that
\begin{align*}
  \int^t_{\ttx} |r_m(s,\xi)|\,ds
  \le \: &
  C|\xi|^{-m+1} \int^\infty_{\ttx} \Xi(s)^{-m}\,ds
\\
% =& CN^{-m+1}\La\(\ttx\)^{m-1} \int^\infty_{\ttx} \Xi(s)^{-m}\,d
%\\
 = \: & CN^{-m}  
  \La\(\ttx\)^{m} \Th(\ttx)^{-1} \int^\infty_{\ttx} \Xi(s)^{-m}\,ds
  \: |\xi|\Th(\ttx) 
\\
 \le \: & C
  N^{-m} 
  \sup_{t\ge 0}\left\{
    \La(t)^{m} \Th(t)^{-1} \int^\infty_{t} \Xi(s)^{-m}\,ds
  \right\}
  |\xi|\Th\(\La^{-1}\(\frac{\tN}{|\xi|}\)\) 
\\
 \le \: & 
  |\xi|\Th\(\La^{-1}\(\frac{N}{|\xi|}\)\). 
\end{align*}
By \eqref{cEV1} and the same way to derive the estimate 
\eqref{estVm}, we have the following estimate:
\begin{equation*}%\label{estcE-tZH}
  \cE(t,\xi)
  \lesssim \exp\(|\xi|\Th\(\La^{-1}\(\frac{N}{|\xi|}\)\)\) \cE(\ttx,\xi)
\end{equation*}
for any $t \ge \ttx$. 
Combining the estimates of $\cE(t,\xi)$ in $\tZ_\Psi$ and $\tZ_H$, we have 
\begin{align*}
  \cE(t,\xi) \lesssim \exp\(2|\xi|\Th\(\La^{-1}\(\frac{N}{|\xi|}\)\)\). 
\end{align*}
\qed

%%%%%%%%%%%%%%%%%%%%%%%%%%%%%%%%%%%%%%%%%%%%%
%%%%%%%%%%%%%%%%%%%%%%%%%%%%%%%%%%%%%%%%%%%%%
%
\section{Proof of Theorem \ref{Thm3}}
%
%%%%%%%%%%%%%%%%%%%%%%%%%%%%%%%%%%%%%%%%%%%%%
%%%%%%%%%%%%%%%%%%%%%%%%%%%%%%%%%%%%%%%%%%%%%

The following proposition is essential for the proof of Theorem \ref{Thm3}. 

\begin{proposition}\label{Prop-Thm3}
Let $\{M_j\}$ be a logarithmically convex sequence. 
If $v=\{v[k]\}_{k\in \Z^d}$ satisfies 
\begin{equation}\label{eq1-Prop-Thm3}
  \sum_{k\in \Z^d} k^\al v[k] = 0 
\end{equation}
for any $\al\in \N_0^d$ and there exists a positive constant $\rho$ such that 
\begin{equation}\label{eq2-Prop-Thm3}
  \sup_{k\in \N^d}\left\{
   |v[k]| |k|^{d+1} T[\{M_j\}]\(\rho^{-1}|k|\) 
  \right\} <\infty, 
\end{equation}
then the following estimate is established: 
\begin{equation*}
  \sup_{\th\in \T^d\setminus\{0\}}
  \left\{
  \left|\hat{v}(\th)\right|T\left[\left\{\frac{M_j}{j!}\right\}\right]
  \(\frac{1}{d\rho|\th|}\)
  \right\}<\infty.
\end{equation*}
\end{proposition}

In order to prove Proposition \ref{Prop-Thm3}, we introduce the following lemma: 
%%%%%%%%%%%%%%%%%%%%%%%%%%%%%%%%
\begin{lemma}\label{Lemma1-Thm3}
If $f\in C^\infty(\T^d)$ satisfies 
$(\pa_\th^\al f)(0)=0$ and 
\begin{equation}\label{eq-Lemma1-Thm3}
  \left|\pa_\th^\al f(\th)\right| \le C M_{|\al|} \rho^{|\al|}
\end{equation}
for any $\al \in \N_0^d$ with positive constants $C$, $\rho$ and a logarithmically convex sequence $\{M_j\}$, then the following estimate is established: 
\begin{equation*}
  \sup_{\th\in \T^d\setminus\{0\}}\left\{
  T\left[\left\{\frac{M_j}{j!}\right\}\right]\(\frac{1}{d \rho|\th|}\)
  |f(\th)|
  \right\}<\infty.
\end{equation*}
\end{lemma}
%%%%%%%%%%%%%%%%%%%%%%%%%%%%%%%%
\begin{proof}
For any $\al\in \N_0^\al$, there exists $\ze \in \T^d$ such that 
$f(\th)=(\pa_\th^\al f)(\ze) \th^\al/\al!$ by Taylor's theorem. 
By \eqref{eq-Lemma1-Thm3} and the inequalities $|\th^\al|\le|\th|^{|\al|}$ and 
$|\al|! \le d^{|\al|}\al!$ we have
\begin{align*}
  |f(\th)|
  \le C \frac{M_{|\al|}\rho^{|\al|}}{\al!} \left|\th^\al\right|
  \le C \frac{M_{|\al|}}{|\al|!} \(d\rho|\th|\)^{|\al|}.
\end{align*}
Therefore, we have 
\begin{align*}
  |f(\th)| \le C\inf_{j\in \N_0}\left\{
  \frac{M_j}{j!}\(d\rho |\th|\)^j\right\}
 =\frac{C}{T\left[\left\{\frac{M_j}{j!}\right\}\right]
  \(\frac{1}{d \rho|\th|}\)}.
\end{align*}
\end{proof}

\noindent
{\it Proof of Proposition \ref{Prop-Thm3}.}\; 
Let $\al\in \N_0^d$. 
By \eqref{eq1-Prop-Thm3}, we have 
\begin{align*}
  \pa_\th^\al\hat{v}(\th)|_{\th=0}
=\sum_{k} (-ik)^\al e^{-ik\cd\th} v[k]\Bigr|_{\th=0}
=(-i)^{|\al|}\sum_{k} k^\al v[k] = 0.
\end{align*}
Let $\al\neq 0$. By \eqref{eq2-Prop-Thm3}, there exists a positive constant $C$ such that
\begin{align*}
  \left|\pa_\th^\al \hat{v}(\th)\right|
  \le \: & \sum_{k} \left|(-ik)^\al e^{-ik\cd\th} v[k]\right|
  \le C \sum_{k\in \Z^d\setminus\{0\}}
    |k|^{-d-1+|\al|}\(T[\{M_j\}]\(\rho^{-1}|k|\)\)^{-1}
\\
  \le \: & C\sum_{k\in \Z^d\setminus\{0\}}
  |k|^{-d-1+|\al|}\(\frac{\(\rho^{-1}|k|\)^{|\al|}}{M_{|\al|}}\)^{-1}
  \simeq M_{|\al|} \rho^{|\al|}.
\end{align*}
Therefore, by Lemma \ref{Lemma1-Thm3} we have
\begin{align*}
  \left|\hat{v}(\th)\right|
  \lesssim 
  \frac{1}{T\left[\left\{\frac{M_j}{j!}\right\}\right]
  \(\frac{1}{d \rho|\th|}\)}. 
\end{align*}
\qed

\noindent
{\it Proof of Theorem \ref{Thm3} {\rm (i)}.}\; 
%Let us prove Theorem \ref{Thm3} by using Theorem \ref{Thm2}. 
By \eqref{La-infty2}, \eqref{eq1-Thm3}, $|\th| \le \pi|\xi(\th)|/2$, 
and applying Proposition \ref{Prop-Thm3} with 
$v=D_j^+ u_0$ $(j=1,\ldots,d)$ and $v=u_1$, we have 
\begin{align*}
  \cE(0,\th)
\simeq \: &
  \left|\hat{u}_1(\th)\right|^2
 +\sum_{j=1}^d\left|\widehat{D_j^+ u_0}(\th)\right|^2
\lesssim
  T\left[\left\{\frac{M_j}{j!}\right\}\right]
  \(\frac{1}{d \rho|\th|}\)^{-2}
\\
\le \: &
  L(N,\{M_j\})^{-2}
  \exp\(-2d\rho|\th|\Th\(\La^{-1}\(\frac{N}{d\rho|\th|}\)\)\)
\\
 = \: &L(N,\{M_j\})^{-2}\exp\(-2N\mu\(\frac{d\rho|\th|}{N}\)\)
\\
\le \: &
  L(N,\{M_j\})^{-2}\exp\(-2N\mu\(\frac{\pi d\rho|\xi(\th)|}{2N}\)\). 
\end{align*}
By Theorem \ref{Thm2} there exists a positive constant $\tN_0 \ge 1$ such that 
\begin{equation}\label{tcEtcE0_tN0}
  \cE(t,\th) \lesssim 
  \exp\(2\tN_0\mu\(\frac{|\xi(\th)|}{\tN_0}\)\) \cE(0,\th). 
\end{equation}
If \eqref{eq1-Thm3} holds for any $N\ge 1$, 
then putting $N_0=\max\{\tN_0,\pi d\rho \tN_0/2\}$ and noting 
\begin{align*}
\tN_0 \mu\(\frac{|\xi|}{\tN_0}\)-N_0\mu\(\frac{\pi d\rho|\xi|}{2N_0}\)
\begin{cases}
\ds{=\tN_0\(1 -\frac{\pi d \rho}{2}\)\mu\(\frac{|\xi|}{\tN_0}\)} 
  \le 0 & \text{ for }\; \dfrac{\pi d\rho}{2} \ge 1,
\\[3ex]
\ds{  \le
  \tN_0 \mu\(\frac{|\xi|}{\tN_0}\)-N_0\mu\(\frac{|\xi|}{N_0}\)
  =0}  &  \text{ for }\; \dfrac{\pi d\rho}{2} < 1,
\end{cases}
\end{align*}
we have $\cE(t,\th) \lesssim L(N_0,\{M_j\})^{-2}$. 
It follows that 
\[
  U(N_0,u_0,u_1)\lesssim (2\pi)^d L(N_0,\{M_j\})^{-2} <\infty, 
\]
and thus \eqref{Ebdd} is established. 
\qed

\noindent
{\it Proof of Theorem \ref{Thm3} {\rm (ii)}.}\; 
By the same way for the proof of (i), there exist positive constants $N_0\ge 1$ and $N_1$ such that $L(N_1,\{M_j\})>0$ and 
\begin{align*}
  U(N_0,u_0,u_1)
  \lesssim \: &
  L(N_1,\{M_j\})^{-2}
\\
  & \: \times  \int_{\T^d}
  \exp\(
   2N_0\mu\(\frac{|\xi(\th)|}{N_0}\)
  -2N_1\mu\(\frac{\pi d\rho|\xi(\th)|}{2N_1}\)\)\,d\th. 
\end{align*}
By \eqref{eq4-Thm3} there exists a positive constant $N$ such that 
\begin{equation*}
  \frac{\mu\(\frac{1}{N}\frac{|\xi(\th)|}{N_0}\)}
       {\mu\(\frac{|\xi(\th)|}{N_0}\)}
 =\frac{\Th\(\La^{-1}\(N \frac{N_0}{|\xi(\th)|}\)\)}
       {N\Th\(\La^{-1}\(\frac{N_0}{|\xi(\th)|}\)\)}
 \ge \frac{N_0}{N_1}. 
\end{equation*}
Therefore, putting $\rho=2N_1/(\pi d N_0 N)$ we have 
\begin{equation*}
  N_0 \mu\(\frac{|\xi(\th)|}{N_0}\)
  \le
  N_1 \mu\(\frac{\pi d \rho|\xi(\th)|}{2N_1}\),
\end{equation*}
it follows that $U(N_0,u_0,u_1)<\infty$, and thus \eqref{Ebdd} is established. 
\qed

%%%%%%%%%%%%%%%%%%%%%%%%%%%%%%%%%%%%%%%%%%%%%
%%%%%%%%%%%%%%%%%%%%%%%%%%%%%%%%%%%%%%%%%%%%%
%
\section{Appendix}
%
%%%%%%%%%%%%%%%%%%%%%%%%%%%%%%%%%%%%%%%%%%%%%
%%%%%%%%%%%%%%%%%%%%%%%%%%%%%%%%%%%%%%%%%%%%%
%%%%%%%%%%%%%%%%%%%%%%%%%%%
\subsection{Proof of Lemma \ref{TDFT}}
%%%%%%%%%%%%%%%%%%%%%%%%%%%
Since $f \in l^1(\Z^d)$, we have
\begin{align*}
  \sum_k e^{-\i k\cd \th} f[k \pm e_j]
=e^{\pm \i e_j \cd \th} \sum_k e^{-\i(k \pm e_j)\cd \th} f[k \pm e_j]
   = e^{\pm \i\th_j} \sum_k e^{-\i k\cd \th}f[k]. 
\end{align*}
Hence we have 
\begin{align*}
  \cF_{\Z^d}[D_j^+ f](\th)
 =\sum_{k} e^{-\i k\cd\th}\(f[k+e_j]-f[k]\)
 =\(e^{\i\th_j} - 1\) \cF_{\Z^d}[f](\th),
\end{align*}
that is, \eqref{TDFT2} is valid. 
Similarly, we have 
\begin{align*}
  \cF_{\Z^d}[D_j^- f](\th)
 =\(1-e^{-\i\th_j}\) \cF_{\Z^d}[f](\th).
\end{align*}
Therefore, we have \eqref{TDFT3} as follows:
\begin{align*}
 \cF_{\Z^d}[D_j^+ D_j^- f](\th)
= \: &\(e^{\i\th_j} - 1\) \cF_{\Z^d}[D_j^- f](\th) 
=\(e^{\i\th_j} - 1\)\(1-e^{-\i\th_j}\) \cF_{\Z^d}[f](\th) 
%\\
% =&\(e^{\i\th_j}-2+e^{-\i\th_j}\) \hat{f}(\th)
\\
= \: &\(e^{\i\frac{\th_j}{2}}-e^{-\i\frac{\th_j}{2}}\)^2 \hat{f}(\th)
 =-4\(\sin\frac{\th_j}{2}\)^2 \hat{f}(\th). 
\end{align*}
\qed

%%%%%%%%%%%%%%%%%%%%%%%%%%%
\subsection{Proof of Lemma \ref{lemm-Parseval}}
%%%%%%%%%%%%%%%%%%%%%%%%%%%
From the definition of $\hat{f}$ we have 
\begin{align*}
  \int_{\T^d}|\hat{f}(\th)|^2\,d\th
= \: &\int_{\T^d}
  \(\sum_{k} e^{-\i k\cd\th} f[k]\)
  \ol{\(\sum_{l} e^{-\i l\cd\th} f[l]\)}\,d\th
%=\int_{\T^d} \sum_{k,l} e^{-\i(k-l)\cd\th} f[k]\ol{f[l]}\,d\th
\\
= \: & \sum_{k,l} f[k]\ol{f[l]} \int_{\T}  e^{-\i(k-l)\cd\th} \,d\th 
\\
= \: &(2\pi)^d \sum_{k,l} \de_{kl} f[k] \ol{f[l]}
=(2\pi)^d \sum_{k} |f[k]|^2, 
\end{align*}
where $\de_{kl}$ denotes the Kronecker delta. 
\qed

%%%%%%%%%%%%%%%%%%%%%%%%%%%
\subsection{Proof of Lemma \ref{lemm-E-cE}}
%%%%%%%%%%%%%%%%%%%%%%%%%%%
By Lemma \ref{TDFT} and Lemma \ref{lemm-Parseval}, we have 
\[
  (2\pi)^d\sum_{k\in\Z^d}|u'(t)[k]|^2
 =\int_{\T^d}\left|\cF_{\Z^d}[u'(t)](\th)\right|^2\,d\th
 =\int_{\T^d}\left|\pa_t \hat{u}(t,\th)\right|^2\,d\th,
\]
and
\begin{align*}
  (2\pi)^d\sum_{j=1}^d \sum_{k\in\Z^d}|D_j^+ u(t)[k]|^2
= \: &\sum_{j=1}^d 
  \int_{\T^d} \left| \cF_{\Z^d}[D_j^+ u(t)](\th)\right|^2\,d\th
\\
= \: &\sum_{j=1}^d 
  \int_{\T^d} \left| \(e^{\i\th_j}-1\)\hat{u}(t,\th)\right|^2\,d\th
\\
= \: &\int_{\T^d} \sum_{j=1}^d \xi_j(\th)^2|\hat{u}(t,\th)|^2\,d\th.
\end{align*}
Thus we have \eqref{EcE}. 
\qed

%%%%%%%%%%%%%%%%%%%%%%%%%%%

\end{document}